\documentclass[12pt,a4paper]{amsart}
\usepackage[english]{babel}
\usepackage{amssymb,latexsym,amsfonts,amsthm,upref,amsmath}
\usepackage[margin=1in]{geometry}
\usepackage[colorlinks=true]{hyperref}
\hypersetup{citecolor=blue, urlcolor=blue, colorlinks=false, linkcolor=red, citecolor=green, filecolor=magenta}

\newtheorem*{thm1}{Theorem}
\newtheorem{thm}{Theorem}
\newtheorem{lem}{Lemma}[subsection]
\newtheorem{prop}{Proposition}[subsection]
\newtheorem{cor}{Corollary}[subsection]

\newtheorem{exm}{Example}[subsection]
\newtheorem{ax}{Remark}[subsection]

\numberwithin{equation}{section}

\begin{document}

\title{COMBINATORIAL BASES Of PRINCIPAL SUBSPACES For Affine Lie ALGEBRA of Type \texorpdfstring{$B_2^{(1)}$}{B_2^{(1)}}}

\author{Marijana Butorac}

\address{University of Rijeka, Department of Mathematics, Radmile Matej\v{c}i\'{c} 2, 51000 Rijeka, Croatia}

\email{mbutorac@math.uniri.hr}



\begin{abstract} 
We consider principal subspaces $W_{L(k\Lambda_0)}$ and $W_{N(k\Lambda_0)}$ of standard
module $L(k\Lambda_0)$ and generalized Verma module $N(k\Lambda_0)$ at level $k\geq 1$ for affine Lie algebra of type $B_2^{(1)}$. By using the theory of vertex ope\-rator algebras, we find combinatorial bases of principal subspaces in terms of quasi-particles. From quasi-particle bases, we obtain character formulas for $W_{L(k\Lambda_0)}$ and $W_{N(k\Lambda_0)}$.
\end{abstract}

\maketitle

\section*{INTRODUCTION}
The principal subspaces of the standard $A_l^{(1)}$-modules play an important role in the study of the Rogers-Ramanujan type identities by 
means of representation theory of affine Lie algebras (see in particular \cite{FS}, \cite{G}, \cite{Cal1}-\cite{Cal2}, \cite{CLM1}- \cite{CLM2}, \cite{CalLM1}-\cite{CalLM3}, \cite{AKS}). 
\par These subspaces were introduced and studied by B. L. Feigin and A. V. Stoyanovsky in \cite{FS}, where they gave a construction of monomial bases of principal subspaces of the standard $A_1^{(1)}$-modules. The bases monomials (which act on a highest weight vector of standard module) can be interpreted in physicist terms, as statistically interacting quasi-particles of color $1$ and charge $1$, whose energies satisfy the difference-two condition (cf. \cite{FS}, \cite{G}, \cite{DKKMM}). From these bases, Feigin and Stoyanovsky obtained the character formulas of the whole standard modules, which are the same as the character formulas of standard $A_1^{(1)}$-modules obtained by J. Lepowsky and M. Primc in \cite{LP}. 
\par This approach was further developed by G. Georgiev in \cite{G}, who constructs the combinatorial bases of  principal subspaces associated to a family of standard $A_l^{(1)}$-modules, in terms of quasi-particles of color $i$, ($1\leq i \leq l$) and charges $\geq 1$, whose energies comply certain difference conditions. In order to prove the linear independence of the spanning sets, Georgiev used the theory of vertex operator algebras, including certain intertwining operators constructed in \cite{DL}. Georgiev's results can be generalized to principal subspaces of standard modules for untwisted affine Lie algebras of types $A$, $D$ and $E$. 
\par In this paper, we use Georgiev's ideas to construct combinatorial bases of principal subspaces of standard modules $L(k{\Lambda}_0)$ and principal subspaces of generalized Verma modules $N(k{\Lambda}_0)$ for affine Lie algebra of type $B_2^{(1)}$. In our case, we neither have a Frenkel-Kac's construction of level $1$ modules nor Dong-Lepowsky's intertwining operators (cf. \cite{DL}). Our main ``instrument'' in the construction of a quasi-particle bases is the theory of vertex o\-perator algebras. Each basis vector is a finite product of  coefficients of certain vertex operators acting on the highest weight vector. Using the vertex algebra structure of $L(k{\Lambda}_0)$ and simple currents, we construct operators which are used in proving linear independence of our bases. 
\par Let $\mathfrak{g}$ be a simple complex Lie algebra of type $B_{2}$, $\mathfrak{h}$ be a Cartan subalgebra
of $\mathfrak{g}$ and $\Pi=\left\{\alpha_{1},\alpha_{2}\right\}$ the set of simple roots. For every root $\alpha$ in the root system $R$, we fix a root vector $x_{\alpha} \in \mathfrak{n}_{\alpha}$, where $\mathfrak{n}_{\alpha}$ is one-dimensional subalgebra of $\mathfrak{g}$. Then we have a triangular decomposition
\begin{equation*}\mathfrak{g}=\mathfrak{n}_{-}\oplus\mathfrak{h}\oplus \mathfrak{n}_{+}.
\end{equation*}
\par Let $\widetilde{\mathfrak{g}}=
\mathfrak{g}\otimes \mathbb{C}[t,t^{-1}]\oplus \mathbb{C}c \oplus \mathbb{C}d$ be the corresponding (untwisted) affine Lie algebra (cf. \cite{K}). For $x \in \mathfrak{g}$ and $m \in \mathbb{Z}$ we write $x(m)=x \otimes t^m$. Denote by $\mathcal{L}({\mathfrak{n}}_{+})$ subalgebra of $\tilde{\mathfrak{g}}$: 
\begin{equation*}\mathcal{L}(\mathfrak{n}_{+})=\mathfrak{n}_{+} \otimes \mathbb{C}[t,t^{-1}].
\end{equation*}
\par As in \cite{G} and \cite{FS}, for fixed positive integer level $k$, we define the principal subspace $W_{L(k{\Lambda}_{0})}$ of standard module $L(k{\Lambda}_{0})$ as
\begin{equation*} W_{L(k{\Lambda}_{0})}=
U(\mathcal{L}(\mathfrak{n}_{+}))v_{k{\Lambda}_{0}},\end{equation*}
where $v_{k{\Lambda}_0}$ is a highest weight vector. Following Georgiev in \cite{G}, we define quasi-particle $x_{r\alpha_i}(m)$ of color $i$, charge $r$ and energy $-m$   
\begin{equation*}
x_{r\alpha_{i}}(m)=\text{Res}_z \left\{ z^{m+r-1}\underbrace{x_{\alpha_{i}}(z) \cdots  x_{\alpha_{i}}(z)}_{\text{r factors}}\right\}
\end{equation*}
 where $x_{\alpha_{i}}(z)=\sum_{j\in \mathbb Z} x_{\alpha_i}(j) z\sp{-j-1}$ are vertex operators associated to elements $x_{\alpha_1}$,  $x_{\alpha_2} \in L(k{\Lambda}_{0})$. Every basis vector is monomial vector $bv_{k{\Lambda}_{0}}$, where monomial $b$ is composed of quasi-particles of colors $i = 1$ and $i = 2$, certain charges $r \geq 1$ and energies 
$-m$ which comply certain difference conditions. 
\par In order to find difference conditions, we use the theory of vertex operator algebras, relations on vertex operator algebra $N(k{\Lambda}_{0})$ and principal subspace $W_{N(k{\Lambda}_{0})} \subset N(k{\Lambda}_{0})$, defined as \begin{equation*}W_{N(k{\Lambda}_{0})}=U(\mathcal{L}(\mathfrak{n}_{+}))v_{N(k{\Lambda}_{0})},\end{equation*}
where $v_{N(k{\Lambda}_{0})}$ is a highest weight vector of $N(k{\Lambda}_{0})$. These relations will lead us to the spanning sets for $W_{L(k{\Lambda}_{0})}$ and $W_{N(k{\Lambda}_{0})}$.
\par The proof of linear independence of spanning sets for $W_{L(k{\Lambda}_{0})}$ is carried out by induction, using a coefficient of a certain  intertwining operator, as well as other operators from the vertex algebra theory. The other important ``ingredient'' in the proof of linear independence is a projection (defined similarly as a projection in \cite{G}) distributing the constituting quasi-particles of monomials among factors of the tensor product of $\mathfrak{h}$-weight subspaces of standard modules, thus ensuring the compatibility of the usage of above mentioned operators with the defining quasi-particles of charge $>1$. By using this proof, we find a basis of $W_{N(k{\Lambda}_{0})}$. Finally, fermionic-type character formulas for $W_{L(k{\Lambda}_{0})}$ and $W_{N(k{\Lambda}_{0})}$ follow as a direct consequence of obtaining the quasi-particle bases. The main result of this paper is the following identity of Rogers-Ramanujan's type, which follows from the character formula for $W_{N(k{\Lambda}_{0})}$:
\begin{thm1}
\begin{align}\nonumber
\prod_{m > 0}& \frac{1}{(1-q^my_1)}\frac{1}{(1-q^my_2)}\frac{1}{(1-q^my_1y_2)}\frac{1}{(1-q^my_1y_2^2)}\\
\nonumber
= \sum_{\substack{r^{(1)}_{1}\geq \ldots \geq  r^{(u)}_{1} \geq  0\\ u \geq  0}}&\frac{q^{r^{(1)^{2}}_{1}+\cdots  +r^{(u)^{2}}_{1}}}{(q)_{r^{(1)}_{1}-r^{(2)}_{1}}\cdots (q)_{r^{(u)}_{1}}}y^{r_1}_{1}\\
\nonumber
\sum_{\substack{r^{(1)}_{2}\geq \ldots \geq r^{(v)}_{2}\geq 0\\v \geq 
0}}&\frac{q^{r^{(1)^{2}}_{2}+ \cdots +r^{(v)^{2}}_{2}-r_{1}^{(1)}(r_{2}^{(1)}+r_{2}^{(2)})
- \cdots -r_{1}^{(u)}(r_{2}^{(2u-1)}
+r_{2}^{(2u)})- \cdots -r_{1}^{(v)}(r_{2}^{(2v-1)}
+r_{2}^{(2v)})}}{(q)_{r^{(1)}_{2}-r^{(2)}_{2}} \cdots (q)_{r^{(v)}_{2}}}y^{r_2}_{2}. 
\end{align}
\end{thm1}
The paper is organized as follows. In Section 1, we introduce most of our notations
and definitions. In Section 2, we build quasi-particle bases for
the principal subspaces of the standard modules and generalized Verma modules. In Section 3, we prove the theorems \ref{thm:1} and \ref{S67T1}. In the last section, we define characters of the principal subspaces $W_{L(k{\Lambda}_{0})}$ and $W_{N(k{\Lambda}_{0})}$ and use the quasi-particle basis to calculate the character formulas directly.

\section{PRINCIPAL SUBSPACE}\label{s:princ}
\subsection{Modules of affine Lie algebra of type \texorpdfstring{$B_2^{(1)}$}{B2(1)}} \label{ss:modules}
Let $\mathfrak{g}$ be a complex simple Lie algebra of type $B_{2}^{(1)}$, let $\mathfrak{h}$ be a Cartan subalgebra
of $\mathfrak{g}$ and $R$ the corresponding root system. Denote by $\left\langle \cdot , \cdot \right\rangle$ the standard symmetric invariant nondegenerate bilinear form on $\mathfrak{g}$ which enables us to identify $\mathfrak{h}$ with its dual $\mathfrak{h}^{*}$. We normalize this form so that $\left\langle \alpha , \alpha \right\rangle = 2$ for every long root $\alpha \in R$. 
\par Let $\epsilon_1,\epsilon_2$ be an orthonormal basis of the real span of the root system $R$ so that $R$ has the following base: $\Pi =\left\{\alpha_1, \alpha_2\right\}$, $\alpha_1=\epsilon_1-\epsilon_2$, $\alpha_2=\epsilon_2$. The set of positive roots is $R_{+}=\left\{\alpha_{1}, \alpha_{1}+\alpha_{2}, \alpha_{1}+2\alpha_{2},\alpha_{2}\right\}$ and the maximal root is $\theta=\alpha_{1}+2\alpha_{2}$. Then we have the triangular decomposition, $\mathfrak{g}=\mathfrak{n}_{-}\oplus \mathfrak{h}\oplus\mathfrak{n}_{+}$. For $\alpha \in R$, we have the corresponding coroot $\alpha^{\vee}=\frac{2 \alpha}{\left\langle \alpha, \alpha\right\rangle}$. Also, for each root $\alpha \in R$ fix a root vector $x_{\alpha} \in \mathfrak{g}$. We will work with one-dimensional subalgebras of $\mathfrak{g}$ $$\mathfrak{n}_{\alpha}= \mathbb{C}x_{\alpha}.$$ 
\par Let $Q=\sum^{2}_{i=1}\mathbb{Z}\alpha_{i}$ and $P=\sum^{2}_{i=1}\mathbb{Z}\omega_{i}$ be the root and weight lattices, 
 where $\omega_1, \omega_2$ are the fundamental weights of $\mathfrak{g}$: $\left\langle \omega_{i} , \alpha_{j}\sp\vee \right\rangle=\delta_{i,j}$, $i,j=1,2$. For later use, we set $\omega_{0}=0$. 
\par We consider the untwisted affine Lie algebra associated to $\mathfrak{g}$,
\begin{equation*}
\widehat{\mathfrak{g}}= \mathfrak{g}\otimes \mathbb{C}[t,t^{-1}]\oplus \mathbb{C}c,
\end{equation*}
where $c$ is a non-zero central element (cf. \cite{K}). For every $x \in\mathfrak{g}$ and $j \in \mathbb{Z}$, we write $x(j)$ for elements $x\otimes t^{j}$. Commutation relations are then given by
\begin{equation*}
\left[c, x(j)\right]=0,
\end{equation*}
\begin{equation*}
\left[x(j_1),y(j_2)\right]= \left[x, y\right](j_1+j_2) + \left\langle x, y \right\rangle j_1 \delta_{j_1+j_2,0}c,
\end{equation*}
where $x(j) = x \otimes t^j$ for any $x, y \in \mathfrak{g}, \ \ j,j_1,j_2 \in \mathbb{Z}$. Let us introduce the following subalgebras of $\widehat{\mathfrak{g}}$
\begin{align*} 
\widehat{\mathfrak{g}}_{\geq 0}=\bigoplus_{n\geq0} \mathfrak{g}\otimes t^{n}\oplus \mathbb{C}c ,\ \ 
\widehat{\mathfrak{g}}_{< 0}=\bigoplus_{n< 0} \mathfrak{g}\otimes t^{n},
\end{align*}
\begin{align*}
 \mathcal{L}(\mathfrak{n}_{+})=\mathfrak{n}_{+} \otimes \mathbb{C}[t,t^{-1}],
\end{align*}
\begin{align*}    
\mathcal{L}(\mathfrak{n}_{+})_{\geq 0}=\mathcal{L}(\mathfrak{n}_{+}) \otimes \mathbb{C}[t], \ \ \mathcal{L}(\mathfrak{n}_{+})_{< 0}=\mathcal{L}(\mathfrak{n}_{+}) \otimes t^{-1}\mathbb{C}[t]
\end{align*}
and 
\begin{equation*}
\mathcal{L}\left(\mathfrak{n}_{\alpha_{i}}\right) =\mathfrak{n}_{\alpha_{i}}\otimes \mathbb{C}[t,t^{-1}],
\end{equation*}
where $\alpha_i \in \Pi$. By adjoining the degree operator $d$ such that
\begin{equation}\label{eq:S11}
\left[d, x (j)\right]=jx(j),\ \    [d, c] = 0
\end{equation}
to the Lie algebra $\widehat{\mathfrak{g}}$, one obtains the affine Kac-Moody algebra \begin{equation*}\widetilde{\mathfrak{g}}= \widehat{\mathfrak{g}} \oplus \mathbb{C}d,\end{equation*} (cf. \cite{K}).
\par Set $\widetilde{\mathfrak{h}}= \mathfrak{h} \oplus  \mathbb{C}c \oplus \mathbb{C}d$. The form $\left\langle  \cdot, \cdot 
\right\rangle$ on $\mathfrak{h}$ extends naturally to $\widetilde{\mathfrak{h}}$. We shall identify $\widetilde{\mathfrak{h}}$ with its dual space $\widetilde{\mathfrak{h}}^*$ via this form. We define $\delta \in \widetilde{\mathfrak{h}}^{\ast}$ by $\delta(d)=1$, $\delta(c)=0$ and $\delta(h)=0$, for every $h \in \mathfrak{h}$. Set $\alpha_0=\delta-\theta$ and $\alpha_0\sp\vee=c-\theta\sp\vee$. Then   $\left\{\alpha_{0},\alpha_{1},\alpha_{2}\right\}$ is a set of simple roots and   $\left\{\alpha_{0}\sp\vee,\alpha_{1}\sp\vee,\alpha_{2}\sp\vee\right\}$ is a set of simple coroots of $\widetilde{\mathfrak{g}}$.  
\par For every simple root $\alpha \in \Pi$, let $sl_{2}(\alpha)=\text{span} \left\{x_{\alpha}, x_{-\alpha}, \alpha\sp\vee\right\}$ be three-dimen\-sional subalgebra of $\mathfrak{g}$ so that the map
\begin{equation*} h\mapsto \alpha\sp\vee, \ \ e \mapsto x_{\alpha}, \ \ f\mapsto x_{-\alpha}\end{equation*} is isomorphism with Lie algebra  $sl_2(\mathbb{C})=\text{span} \left\{e,f,h\right\}$ with commutation relations:
\begin{equation*}
\left[h,e\right]=2e,\ \left[h,f\right]=-2f,\ \left[e,f\right]=h.
\end{equation*} 
The associated affine Lie algebra $\widetilde{\mathfrak{sl}}_{2}(\alpha)$, in this case, is subalgebra of $\widetilde{\mathfrak{g}}$ with  the central element 
\begin{equation}\label{eq:S12}
c'=\frac{2c}{\left\langle \alpha,\alpha\right\rangle}.
\end{equation}
\par Define fundamental weights of $\widetilde{\mathfrak{g}}$ by $\left\langle \Lambda_{i} , \alpha_{j}\sp\vee \right\rangle=\delta_{i,j}$ for $i,j=0,1,2$ and $\Lambda_i\left(d\right)=0$. Denote by $L(\Lambda_0)$, $L(\Lambda_1)$, $L(\Lambda_2)$ standard $\widetilde{\mathfrak{g}}$-modules of level $1$, that is integrable highest weight $\widetilde{\mathfrak{g}}$-modules of level 1 with highest weights $\Lambda_0$, $\Lambda_1$, $\Lambda_2$ and with highest weight vectors $v_{\Lambda_0}, v_{\Lambda_1}, v_{\Lambda_2}$.  
 Note that $L(\Lambda_{0})$ is a direct sum of standard $\widetilde{\mathfrak{sl}}_{2}(\alpha)$-modules of level $1$ if $\alpha$ is a long root and level $2$ if $\alpha$ is a short root. This follows from (\ref{eq:S12}) (cf. \cite{K}). 
\par For the sake of simplicity, we shall restrict our investigation to the $\widehat{\mathfrak{g}}$-module $N(k\Lambda_{0})$ and its irreducible quotient $L(k\Lambda_{0})$, where level $k$ is a positive integer. Throughout this paper, we will write $x(m)$ for the action of $x\otimes t^m$ on any $\widehat{\mathfrak{g}}$-module, where $x \in \mathfrak{g}$ and $j \in \mathbb{Z}$.
\par The generalized Verma module $N(k\Lambda_{0})$ is defined as the induced $\widehat{\mathfrak{g}}$-module
\begin{equation*} 
N(k\Lambda_{0})= 
U(\widehat{\mathfrak{g}})\otimes_{U(\widehat{\mathfrak{g}}_{\geq 0})} \mathbb{C}v_{N(k\Lambda_{0})},
\end{equation*}
where $\mathbb{C}v_{N(k\Lambda_{0})}$ is 1-dimensional $\widehat{\mathfrak{g}}_{\geq 0}$-module, such that \begin{equation*}cv_{N(k\Lambda_{0})}=kv_{N(k\Lambda_{0})}\end{equation*} and \begin{equation*}(\mathfrak{g}\otimes t^{j})v_{N(k\Lambda_{0})}=0,\end{equation*} for every $j\geq 0$. From the Poincar\'{e}-Birkhoff-Witt theorem, we have  
\begin{equation*} 
N(k\Lambda_{0})\cong  U(\widehat{\mathfrak{g}}_{<0})\otimes_{\mathbb{C}} \mathbb{C}v_{N(k\Lambda_{0})}
\end{equation*} 
as vector spaces. Set \begin{equation*}v_{N(k\Lambda_{0})}=1 \otimes v_{N(k\Lambda_{0})}.\end{equation*}
We view $\widehat{\mathfrak{g}}$-modules $N(k\Lambda_{0})$ and $L(k\Lambda_{0})$ as $\widetilde{\mathfrak{g}}$-modules, where $d$ acts as \begin{equation}\label{eq:S13}
dv_{N(k\Lambda_{0})}=0
\end{equation} (see \cite{LL}).

\subsection{Principal subspace}
Let $k \in \mathbb{N}$ and let $\Lambda = k \Lambda_0$, the only dominant integral weight of level k which we consider. Set $v_{k\Lambda_0}$ to be the highest weight vector of the standard module $L(k\Lambda_0)$. As in \cite{FS} and \cite{G}, we define the principal subspace $W_{L(k\Lambda_0)}$ of the standard module $L(k\Lambda_0)$ as
\begin{equation*}
    W_{L(k\Lambda_0)}= U(\mathcal{L}(\mathfrak{n}_+))v_{k\Lambda_0}.
\end{equation*}
Here $U(\mathcal{L}(\mathfrak{n}_+))$ is the universal enveloping algebra of Lie algebra $\mathcal{L}(\mathfrak{n}_+)$. Continuing to generalize \cite{FS} and \cite{G}, we introduce the principal subspace $W_{N(k\Lambda_{0})}$ of the generalized Verma module $N(k\Lambda_{0})$ as
\begin{equation*}
W_{N(k\Lambda_{0})}= U(\mathcal{L}(\mathfrak{n}_{+}))v_{N(k\Lambda_{0})}.
\end{equation*}
\par Like in \cite{G}, we denote the vector space 
\begin{equation*}
U = U(\mathcal{L}\left(\mathfrak{n}_{\alpha_{2}}\right))U(\mathcal{L}\left(\mathfrak{n}_{\alpha_{1}}\right)).\end{equation*}
\par Using the same argument as Georgiev in \cite{G}, we have
\begin{lem}
\begin{align}\nonumber
W_{L(k\Lambda_{0})}& = Uv_{k\Lambda_{0}},\\
\nonumber
W_{N(k\Lambda_{0})}&=Uv_{N(k\Lambda_{0})}.
\end{align}
\end{lem} 
\begin{flushright}
$\square$
\end{flushright}

\subsection{Vertex operator algebra structure for \texorpdfstring{$N(k\Lambda_0)$}{N(kLambda0)} and \texorpdfstring{$L(k\Lambda_0)$}{L(kLambda0)}} 
It is known that, for every positive integer $k$, the generalized Verma module $N(k\Lambda_{0})$ has a structure of vertex operator algebra (see \cite{LL}, \cite{Li1}, \cite{MP}), where $v_{N(k\Lambda_{0})}$ is the vacuum vector and
\begin{equation*}Y(x(-1)v_{N(k\Lambda_{0})}, z)=x(z)\end{equation*}
is vertex operator associated with the vector $x(-1)v_{N(k\Lambda_{0})} \in N(k\Lambda_{0})$. Denote by $I(k\Lambda_{0})$ the sum of all  the proper ideals in the vertex operator algebra $N(k\Lambda_{0})$. The simple vertex algebra $N(k\Lambda_{0})/I(k\Lambda_{0})$ is then the standard $\widehat{\mathfrak{g}}$-module, that is 
\begin{equation*}N(k\Lambda_{0})/I(k\Lambda_{0})\cong L(k\Lambda_{0})\end{equation*}
(cf. \cite{LL}, \cite{MP}). In addition, all the level $k$ standard modules are modules for this vertex operator algebra (cf. \cite{LL}, \cite{MP}).
\par Standard modules of level $k > 1$ can be viewed, by complete reducibility, as submo\-dules of tensor products of standard modules of level $1$. Vertex operators $x(z)$, where $x \in \mathfrak{g}$, act on the tensor product of standard modules of level $1$ as Lie algebra elements (cf. \cite{LL}).
\par For $\alpha \in R$ and $r>0$, we have 
\begin{equation*} 
 Y(\left(x_{\alpha}(-1)\right)^rv_{N(k\Lambda_{0})},z) = x_{\alpha}(z)^r.
 \end{equation*}
 \par We will use the commutator formula: 
\begin{align}\label{al:S14} 
[Y(x_{\alpha}(-1)v_{N(k\Lambda_{0})},z_1), Y(x_{\beta}(-1)^rv_{N(k\Lambda_{0})},z_2)]\\
\nonumber
= \sum_{j \geq 0} \frac{(-1)^j}{j!} \left(\frac{d}{dz_1} 
 \right)^j z^{-1}_2
\delta\left(\frac{z_1}{z_2}\right)Y(x_{\alpha}(j)x_{\beta}(-1)^rv_{N(k\Lambda_{0})},z_2),
\end{align}
where $\alpha, \beta \in R$, (cf. \cite{FHL}). 
 \par Let $k \in \mathbb{N}$. Then we have the following relations on the standard module $L(k\Lambda_{0})$:
\begin{prop}[cf. \cite{LL}, \cite{Li1}, \cite{MP}]
\begin{align}\label{al:S15}
x_{\alpha_1}(z)^{k+1}&=0,\\
\label{al:S16}
x_{\alpha_2}(z)^{2k+1}&=0.
\end{align}
\end{prop}

\section{Quasi-particle bases}

\subsection{Quasi-particle monomials}
In the description of our bases, we use quasi-particles as in \cite{G}. In this section, we introduce the concepts of quasi-particles of colors  $i=1,2$ and charges $r\geq 1$, the concepts of monochromatic and polychromatic monomials and the linear order among quasi-particle monomials. We start with the definition of quasi-particles.\\ 
\textbf{Definition of quasi-particles.} For given $i \in \left\{1,2\right\}$, $r \in \mathbb{N}$ and $m\in \mathbb{Z}$ define a quasi-particle 
of color $i$, charge $r$ and energy $-m$ by
\begin{equation}\label{eq:S21}
x_{r\alpha_{i}}(m)=\textup{Res}_z \left\{ z^{m+r-1}\underbrace{x_{\alpha_{i}}(z) \cdots x_{\alpha_{i}}(z)}_{\textup{r factors}}\right\}
\end{equation}
We shall say that vertex operator $x_{r\alpha_{i}}(z)$ represents the generating function of color $i$ and charge $r$. From (\ref{eq:S21}) 
it follows
\begin{equation*} 
x_{r\alpha_{i}}(m)=\sum_{\substack{m_{1},\ldots, m_{r}\in \mathbb{Z} \\ m_{1}+\cdots +m_{r}=m}}x_{\alpha_{i}}(m_{r})\cdots x_{\alpha_{i}}(m_{1}).
\end{equation*}
Note that a family  of operators \begin{equation*}\left(x_{\alpha_{i}}(m_{r})\cdots  x_{\alpha_{i}}(m_{1})\right)_{\substack{m_{1},\ldots ,m_{r}\in
\mathbb{Z} \\ m_{1}+ \cdots + m_{r}=m}}\end{equation*} on the highest weight module is a summable family (cf. \cite{LL}). We will usually denote a product of quasi-particles of color $i$ by $b(\alpha_i)$. We say that monomial $b(\alpha_i)$ is a monochromatic monomial colored with color-type $r_i$, if the sum of all quasi-particle charges in monomial $b(\alpha_i)$ is $r_i$. 
\par Now, choose a partition of positive integer $r_{i}$, which we denote by \begin{equation*}(r^{(1)}_{i}, r^{(2)}_{i}, \ldots , r^{(s)}_{i}),\end{equation*} so that
\begin{equation*} r^{(1)}_{i} \geq  r^{(2)}_{i} \geq \cdots \geq  r^{(s)}_{i} \geq  0 \ \ \text{and} \ \ \ s \geq 1.\end{equation*} 
Denote by 
\begin{equation*}\left(n_{r_{i}^{(1)},i}, \ldots , n_{1,i}\right),\end{equation*}
the conjugate of $(r^{(1)}_{i}, r^{(2)}_{i}, \ldots , r^{(s)}_{i})$ (cf. \cite{A}),
 where
\begin{equation*} 0 \leq n_{r_{i}^{(1)},i} \leq \cdots \leq  n_{1,i}.\end{equation*}
We say that a monochromatic quasi-particle monomial 
\begin{equation*} b(\alpha_i)=x_{n_{r_{i}^{(1)},i}\alpha_{i}}(m_{r_{i}^{(1)},i})\cdots  x_{n_{1,i}\alpha_{i}}(m_{1,i}),\end{equation*}
is of color-type $r_i$, charge-type
\begin{equation*}\left(n_{r_{i}^{(1)},i}, \ldots ,  n_{1,i}\right)\end{equation*}
and dual-charge-type
\begin{equation*} \left(r^{(1)}_{i}, r^{(2)}_{i}, \ldots , r^{(s)}_{i}\right).\end{equation*}
This quasi-particle monomial is built out of $r^{(1)}_{i}-r^{(2)}_{i}$ quasi-particles of charge $1$, $r^{(2)}_{i}-r^{(3)}_{i}$   quasi-particles of charge $2$,$\ldots$, $r^{(s)}_{i}$ quasi-particles of charge $s$. In the demonstration of charge-type and   dual-charge-type of monomial $b(\alpha_i)$ we use the graphic presentation from \cite{G}. 
\par Since quasi-particles of the same color commute, we arrange quasi-particles of the same color and the same charge so 
 that the values $m_{p,i}$, for $1 \leq p \leq  r_i^{(1)}$, form a decreasing sequence of integers from right to left (that is, we start with the highest integer in the given sequence and we finish with the smallest integer).
\par We say that monomials ``colored'' with more colors are polychromatic monomials. For (polychromatic) monomial 
\begin{equation*}b= b(\alpha_{2})b(\alpha_{1}),\end{equation*}
\begin{equation*}=x_{n_{r_{2}^{(1)},2}\alpha_{2}}(m_{r_{2}^{(1)},2}) \cdots  x_{n_{1,2}\alpha_{2}}(m_{1,2}) 
x_{n_{r_{1}^{(1)},1}\alpha_{1}}(m_{r_{1}^{(1)},1}) \cdots  x_{n_{1,1}\alpha_{1}}(m_{1,1}),\end{equation*} 
we will say it is of charge-type 
\begin{equation*}\left(n_{r_{2}^{(1)},2}, \ldots ,n_{1,2};n_{r_{1}^{(1)},1}, \ldots ,n_{1,1}\right),\end{equation*} where
\begin{equation*} 
0 \leq n_{r_{i}^{(1)},i}\leq \ldots \leq  n_{1,i},
\end{equation*}
dual-charge-type
\begin{equation*} \left(r^{(1)}_{2},\ldots , r^{(s_{2})}_{2};r^{(1)}_{1},\ldots , r^{(s_{1})}_{1} \right),\end{equation*}
where
\begin{equation*} 
r^{(1)}_{i}\geq r^{(2)}_{i}\geq \ldots \geq  r^{(s_{i})}_{i}\geq 0 
\end{equation*}
and color-type
\begin{equation*} \left(r_{2},r_{1}\right),\end{equation*}
where 
\begin{equation*} 
r_i=\sum_{p=1}^{r_{i}^{(1)}}n_{p,i}=\sum^{s_{i}}_{t=1}r^{(t)}_{i} \ \ \text{and} \ \ s_{i}\in \mathbb{N},
\end{equation*}
if for every color \begin{equation*}\left(n_{r_{i}^{(1)},i}, \ldots ,n_{1,i}\right)$$ and $$\left(r^{(1)}_{i}, r^{(2)}_{i}, \ldots , 
r^{(s)}_{i}\right)\end{equation*}
are mutually conjugate partitions of $r_i$ (cf. \cite{G}). 
 In this case we can visualize charge-type and dual-charge type of polychromatic monomials $b=b(\alpha_2)b(\alpha_1)$ in 
graphic presentation, as in the Example \ref{Example}.
\par We use the same terminology for the products of generating functions. For the product
\begin{equation*}x_{n_{r_{2}^{(1)},2}\alpha_{2}}(z_{r_{2}^{(1)},2})\cdots  x_{n_{1,2}\alpha_{2}}(z_{1,2}) 
x_{n_{r_{1}^{(1)},1}\alpha_{1}}(z_{r_{1}^{(1)},1})\cdots  x_{n_{1,1}\alpha_{1}}(z_{1,1})\end{equation*}
we say that it is of charge-type \begin{equation*}\left(n_{r_{2}^{(1)},2}, \ldots ,n_{1,2};n_{r_{1}^{(1)},1}, \ldots ,n_{1,1}\right)\end{equation*} and  dual-charge-type
\begin{equation*}\left(r^{(1)}_{2},\ldots , r^{(s_{2})}_{2};r^{(1)}_{1},\ldots , r^{(s_{1})}_{1} 
\right).\end{equation*}
\begin{exm}\label{Example}
For color-type $(r_2,r_1)=(10; 12)$, charge-type $(1,2,3,4;2,3,3,4)$ and dual-charge-type 
$(4,3,2,1; 4,4,3,1)$ we have the following graphic presentation:
\end{exm}
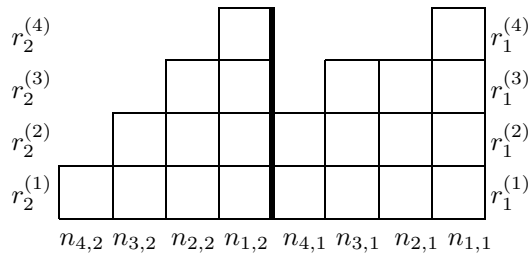
\begin{figure}[h!tb]
\setlength{\unitlength}{7mm}
\begin{center}
\begin{picture}(8,4)
\linethickness{0.1mm}
\multiput(0,0)(1,0){8}
{\line(0,1){1}}
\multiput(0,0)(1,0){8}
{\line(1,0){1}}
\multiput(0,1)(1,0){8}
{\line(1,0){1}}
\multiput(1,2)(1,0){7}
{\line(1,0){1}}
\multiput(1,2)(1,0){3}
{\line(1,0){1}}
\multiput(1,0)(1,0){2}
{\line(0,1){2}}
\multiput(2,0)(1,0){3}
{\line(0,1){3}}
\multiput(2,3)(1,0){2}
{\line(1,0){1}}
\multiput(3,0)(1,0){2}
{\line(0,1){4}}
\multiput(3,4)(1,0){1}
{\line(1,0){1}}
\linethickness{0.75mm}
\multiput(4,0)(1,0){1}
{\line(0,1){4}}
\linethickness{0.1mm}
\multiput(4,1)(1,0){1}
{\line(1,0){4}}
\multiput(5,0)(1,0){1}
{\line(0,1){3}}
\multiput(5,1)(1,0){1}
{\line(1,0){3}}
\multiput(8,0)(1,0){1}
{\line(0,1){4}}
\multiput(6,0)(1,0){2}
{\line(0,1){3}}
\multiput(7,0)(1,0){2}
{\line(0,1){4}}
\multiput(5,3)(1,0){3}
{\line(1,0){1}}
\multiput(7,4)(1,0){1}
{\line(1,0){1}}
\put(-0.9,0.3){\footnotesize{$r_2^{(1)}$}}
\put(-0.9,1.3){\footnotesize{$r_2^{(2)}$}}
\put(-0.9,2.3){\footnotesize{$r_2^{(3)}$}}
\put(-0.9,3.3){\footnotesize{$r_2^{(4)}$}}
\put(8.1,0.3){\footnotesize{$r_1^{(1)}$}}
\put(8.1,1.3){\footnotesize{$r_1^{(2)}$}}
\put(8.1,2.3){\footnotesize{$r_1^{(3)}$}}
\put(8.1,3.3){\footnotesize{$r_1^{(4)}$}}
\put(0.0,-0.5){\footnotesize{$n_{4,2}$}}
\put(1.0,-0.5){\footnotesize{$n_{3,2}$}}
\put(2.1,-0.5){\footnotesize{$n_{2,2}$}}
\put(3.1,-0.5){\footnotesize{$n_{1,2}$}}
\put(4.2,-0.5){\footnotesize{$n_{4,1}$}}
\put(5.2,-0.5){\footnotesize{$n_{3,1}$}}
\put(6.3,-0.5){\footnotesize{$n_{2,1}$}}
\put(7.3,-0.5){\footnotesize{$n_{1,1}$}}
\end{picture}
\bigskip
\caption{Graphic presentation}
\end{center}
\end{figure}
\bigskip
\bigskip
\par We have a ``color-type'' gradation of the entire vector space $U$
\begin{equation}\label{eq:S22}
U=\bigoplus_{r_2,r_1 \geq 0}U_{(r_2,r_1)},
\end{equation}
where $U_{(r_2,r_1)}$ is the weight subspace of weight $r_2 \alpha_2+r_1 \alpha_1$.
\par We compare the (polychromatic) monomials as in \cite{G}; first we compare their charge-types and if the charge-types are the same, we compare their sequences of energies. That is, we apply the following definition (\ref{eq:S23}), starting from color $i=1$:\\ 
for polychromatic monomials, we state 
 \begin{equation}\label{eq:S23}
b< \overline{b}
\end{equation}
if one of the following conditions holds:
\begin{enumerate}
    \item
$\left(n_{r_{2}^{(1)},2}, \ldots  ,n_{1,1}\right)<
\left(\overline{n}_{\overline{r}_{2}^{(1)},2}, \ldots ,\overline{n}_{1,1}\right)$\\ (if there is $u \in \mathbb{N}$, such that   $n_{1,i}=\overline{n}_{1,i}, n_{2,i}=\overline{n}_{2,i},\ldots , n_{u-1,i}=\overline{n}_{u-1,i},$ and 
$u=\overline{r}_{i}^{(1)}+1 \ \ \text{or} \ \ n_{u,i}<\overline{n}_{u,i}$);
\item $\left(n_{r_{2}^{(1)},2}, \ldots , n_{1,1}\right)=
\left(\overline{n}_{\overline{r}_{2}^{(1)},2}, \ldots , \overline{n}_{1,1}\right)$,\\ 
$\left(m_{r_{2}^{(1)},2},\ldots , m_{1,1}\right)
<
\left(\overline{m}_{\overline{r}_{2}^{(1)},2}, \ldots
,\overline{m}_{1,1}\right)$\\
(if the\-re is $u \in \mathbb{N}$, $1\leq u \leq r_i$, such that $m_{1,i}=\overline{m}_{1,i},
m_{2,j}=\overline{m}_{2,j},
\ldots m_{u-1,i}=\overline{m}_{u-1,i}$ and $m_{u,i}<\overline{m}_{u,i}$).
\end{enumerate}

\subsection{Relations among quasi-particles of the same colors}
Relations (\ref{al:S15}) and (\ref{al:S16}) will not be sufficient to determine a basis of $W_{L(k\Lambda_{0})}$. In the next section, we  will determine the remaining relations. Here we establish relations among quasi-particles of the same colors, that is, expressions for the  products of the form $x_{n\alpha}(m)x_{n'\alpha}(m')$, where $\alpha=\alpha_i$, $n,n' \in \mathbb{N}$ and $m,m' \in \mathbb{Z}$. These relations are also used in \cite{F}, \cite{FS} and \cite{G}. In our case, these relations are provided in Lemma \ref{lem:S21} and Corollary \ref{cor:S21}. The proof of these relations may be found in \cite{JP}, so we omit the proof of Lemma \ref{lem:S21} and Corollary \ref{cor:S21}.

\begin{lem}\label{lem:S21}
For fixed $M,j\in \mathbb{Z}$ and $1\leq n\leq n'$ a sequence of $2n$ monomials from the set
\begin{multline*} A=\{x_{n\alpha}(j)x_{n'\alpha}(M-j),x_{n\alpha}(j+1)x_{n'\alpha}(M-j-1),
\ldots ,\\
 \ldots , x_{n\alpha}(j+2n-1)x_{n'\alpha}(M-j-2n+1)\}\end{multline*}
can be expressed as a linear combination of monomials from the set \begin{equation*} \left\{x_{n\alpha}(m)x_{n'\alpha}(m'):m+m'=M\right\} \setminus A\end{equation*} and monomials which have as a factor quasi-particle $x_{(n'+1)\alpha}(j')$, $j' \in \mathbb{Z}$.
\end{lem}
\begin{flushright}
\hfill
\end{flushright}
The special case of Lemma \ref{lem:S21} is distinguished in the next corollary. 
\begin{cor}\label{cor:S21}
Fix $n\in \mathbb{N}$ and $j \in \mathbb{Z}$. The elements from the set \begin{equation*}A_1=\{x_{n\alpha}(m)x_{n\alpha}(m'):m'-2n< m \leq   m'\}\end{equation*}
can be expressed as a linear combination of monomials $x_{n\alpha}(m)x_{n\alpha}(m')$, such that $$m\leq m'-2n$$ 
and monomials with quasi-particle $x_{(n+1)\alpha_i}(j')$, $j' \in \mathbb{Z}$.
\end{cor}

\subsection{Relations among quasi-particles of different colors}
Here we establish relations among quasi-particles of different colors, that is expressions for the products of the form  $x_{n_i\alpha_i}(m_i)x_{n_j'\alpha_j}(m_j)$, where $i,j =1,2$, $i\neq j$ $n_i,n_j' \in \mathbb{N}$ and $m_i,m'_j \in \mathbb{Z}$.
\begin{lem}\label{lem:S231}
For fixed $n_1\in \mathbb{N}$, we have:
  \end{lem}
\begin{itemize}
	\item [a)]
$x_{\alpha_{2}}(0)x_{\alpha_{1}}(-1)^{n_1}v_{N(k\Lambda_{0})}=n_1x_{\alpha_{1}}(-1)^{n_1-1}x_{\alpha_{1}+\alpha_{2}}(-1)v_{N(k\Lambda_{0})}$;
\item[b)] For $j>0$, we have $x_{\alpha_{2}}(j)x_{\alpha_{1}}(-1)^{n_1}v_{N(k\Lambda_{0})}=0$.
\end{itemize}
\begin{proof}
In the proof of the statments a) and b), we use commutation relations 
\begin{align}\label{al:S24}
[x_{\alpha_{2}}, x_{\alpha_{1}}]&=x_{\alpha_{1} + \alpha_{2}},\\
\label{al:S25}
[x_{\alpha_{2}}, x_{\alpha_{1} +\alpha_{2}}]&=x_{\alpha_{1} + 2\alpha_{2}}
\end{align} 
and induction on $n_1\in \mathbb{N}$.
\end{proof}
We use Lemma \ref{lem:S231} and the commutator formula for vertex operators \ref{al:S14} in the proof of the following lemma:
\begin{lem}\label{lem:S232}
Let $n_1,n_2 \in \mathbb{N}$ be fixed. One has
\begin{equation}\label{eq:S26}
(z_{2}-z_{1})^{n_2}x_{n_1\alpha_{1}}(z_{1})x_{n_2\alpha_{2}}(z_{2})=(z_{2}-z_{1})^{n_2}x_{n_2\alpha_{2}}(z_{2})x_{n_1\alpha_{1}}(z_{1}).
\end{equation}
\end{lem}
\begin{proof}
From 
\begin{equation*}[Y(x_{\alpha_{2}}(-1)v_{N(k\Lambda_0)},z_2), Y(x_{\alpha_{1}}(-1)^{n_1}v_{N(k\Lambda_0)},z_1)]
=z^{-1}_1
\delta\left(\frac{z_2}{z_1}\right)Y(x_{\alpha_{2}}(0)x_{\alpha_{1}}(-1)^{n_1}v_{N(k\Lambda_0)},z_1)
\end{equation*}
and properties of $\delta$-function (cf. \cite{LL}) follows
\begin{equation*}
(z_{2}-z_{1})x_{\alpha_{2}}(z_{2})x_{n_1\alpha_{1}}(z_{1})=(z_{2}-z_{1})x_{n_1\alpha_{1}}(z_{1})x_{\alpha_{2}}(z_{2}).
\end{equation*}
Therefore, one has
\begin{equation*}
\prod_{j=1}^{n_2}(z_{j,2}-z_1)x_{\alpha_{2}}(z_{1,2})\cdots  x_{\alpha_{2}}(z_{n_2,2})x_{n_1\alpha_{1}}(z_{1})=\prod_{j=1}^{n_2}(z_{j,2}-z_1)x_{n_1\alpha_{1}}(z_{1})x_{\alpha_{2}}(z_{1,2})\cdots       x_{\alpha_{2}}(z_{n_2,2}).
\end{equation*}
For every $j$, $1\leq j \leq n_2$, using $\text{lim}_{z_{j,2}\rightarrow z_2}$, follows  
(\ref{eq:S26}).
\end{proof}
Just like in Lemma \ref{lem:S231}, from commutation relations (\ref{al:S24}) and (\ref{al:S25}) and induction on $n_2 \in \mathbb{N}$  follows:
\begin{lem}\label{lem:S233} 
For fixed $n_2 \in \mathbb{N}$, we have
\end{lem}
\begin{itemize}
\item[a)]
$ x_{\alpha_{1}}(0)x_{\alpha_{2}}(-1)^{n_2}v_{N(k\Lambda_0)}=-n_2x_{\alpha_{2}}(-1)^{n_2-1}x_{\alpha_{1}
 +\alpha_{2}}(-1)v_{N(k\Lambda_0)}$ 
\begin{equation*}
 +\frac{n_2(n_2-1)}{2}x_{\alpha_{2}}(-1)^{n_2-2}x_{\alpha_{1}+2\alpha_{2}}(-2)v_{N(k\Lambda_0)}\end{equation*}
\item[b)] 
$x_{\alpha_{1}}(1)x_{\alpha_{2}}(-1)^{n_2}v_{N(k\Lambda_0)}=\frac{n_2(n_2-1)}{2}x_{\alpha_{2}}(-1)^{n_2-2}x_{\alpha_{1}+2\alpha_{2}}(-1)v_{N(k\Lambda_0)}$
\item[c)] For every $j\geq 2$, we have $x_{\alpha_{1}}(j)x_{\alpha_{2}}(-1)^{n_2}v_{N(k\Lambda_0)}=0$.
\end{itemize}
\begin{flushright}
$\square$
\end{flushright}

\begin{lem}\label{lem:S234}
Let $n_1,n_2 \in \mathbb{N}$ be fixed. One has
\begin{equation}\label{eq:S27}
(z_{2}-z_{1})^{2n_1}x_{n_1\alpha_{1}}(z_{1})x_{n_2\alpha_{2}}(z_{2})=(z_{2}-z_{1})^{2n_1}x_{n_2\alpha_{2}}(z_{2})x_{n_1\alpha_{1}}(z_{1}).
\end{equation}
\end{lem}
\begin{proof}
Like in Lemma \ref{lem:S232}, from Lemma \ref{lem:S233}, commutator formula for vertex operators 
\begin{equation*}[Y(x_{\alpha_{1}}(-1)v_{N(k\Lambda_0)},z_1), Y(x_{\alpha_{2}}(-1)^{n_2}v_{N(k\Lambda_0)},z_2)]\end{equation*}
 \begin{equation*}= z^{-1}_2
\delta(\frac{z_1}{z_2})Y(x_{\alpha_{1}}(0)x_{\alpha_{2}}(-1)^{n_2}v_{N(k\Lambda_0)},z_2)- \left(\frac{d}{dz_1}
 \right)z^{-1}_2
\delta(\frac{z_1}{z_2})Y(x_{\alpha_{1}}(1)x_{\alpha_{2}}(-1)^{n_2}v_{N(k\Lambda_0)},z_2).\end{equation*}
and properties of $\delta$-function follows
\begin{equation*}(z_{2}-z_{1})^{2}x_{\alpha_{1}}(z_{1})x_{n_2\alpha_{2}}(z_{2})=(z_{2}-z_{1})^{2}x_{n_2\alpha_{2}}(z_{2})x_{\alpha_{1}}(z_{1}).\end{equation*}
Now, we have
\begin{equation*}\prod_{j=1}^{n_1}(z_{j,1}-z_1)^2x_{\alpha_{1}}(z_{1,1})
\cdots    x_{\alpha_{1}}(z_{n_1,1})x_{n_2\alpha_{2}}(z_{2})=\prod_{j=1}^{n_1}(z_{2}-z_{j,1})^2x_{n_2\alpha_{2}}(z_{2})x_{\alpha_{1}}(z_{1,1})\cdots  
 x_{\alpha_{1}}(z_{n_1,1}).\end{equation*}  
For every $j$, $1\leq j \leq n_1$, using limit $\text{lim}_{z_{j,1}\rightarrow z_1}$, we get (\ref{eq:S27}).
\end{proof}

\begin{lem}\label{lem:S235}
For generating function  
\begin{equation*}x_{n_{r^{(1)}_{2},2}\alpha_{2}}(z_{r^{(1)}_{2},2})\cdots  x_{n_{1,1}\alpha_{1}}(z_{1,1})\end{equation*} of charge-type  
$\left(n_{r_{2}^{(1)},2},\ldots , n_{1,2};n_{r_{1}^{(1)},1},\ldots ,n_{1,1}\right)$ and dual-charge-type 
$\left(r_{2}^{(1)},r_{2}^{(2)},\ldots ;\right.$\\ $\left. r_{1}^{(1)}, r_{1}^{(2)}, \ldots \right)$, one has
\begin{align}\nonumber
\left[\prod^{r^{(1)}_{2}}_{p=1}\prod^{r^{(1)}_{1}}_{q=1}\left(1-\frac{z_{q,1}}{z_{p,2}}\right)^{\emph{\text{min}}\left\{ 
n_{p,2},2n_{q,1}\right\}} \right] x_{n_{r^{(1)}_{2},2}\alpha_{2}}(z_{r^{(1)}_{2},2}) \cdots         x_{n_{1,1}\alpha_{1}}(z_{1,1})v_{N(k\Lambda_{0})}\\
\label{al:S28} 
\in \left[\prod^{r^{(1)}_{2}}_{p=1} 
z_{p,2}^{-\sum^{r^{(1)}_{1}}_{q=1}\emph{\text{min}}\left\{ 
n_{p,2},2n_{q,1}\right\}}\right]W_{N(k\Lambda_{0})}\left[\left[z_{r^{(1)}_{2},2},\ldots ,z_{1,1}\right]\right]. 
\end{align}
\end{lem}
\begin{proof}
(\ref{al:S28}) follows from 
\begin{align}\nonumber
\prod^{r^{(1)}_{2}}_{p=1}\prod^{r^{(1)}_{1}}_{q=1}(z_{p,2}-z_{q,1})^{\text{min}\left\{2n_{q,1},n_{p,2}\right\}}x_{n_{r^{(1)}_{2},2}\alpha_{2}}(z_{r^{(1)}_{2},2})\cdots   x_{n_{1,1}\alpha_{1}}(z_{1,1})v_{N(k{\Lambda}_{0})}\\
\label{al:S29}
\in W_{N(k\Lambda_{0})}\left[\left[z_{r^{(1)}_{2},2},\ldots , z_{1,1}\right]\right].
\end{align}
(\ref{al:S29}) is immediate from creation property of vertex operators (cf. \cite{LL}), Lemma \ref{lem:S232} and Lemma \ref{lem:S234}, i.e.
\begin{equation*}(z_{2}-z_{1})^{\text{min}\left\{2n_1,n_2\right\}}x_{n_1\alpha_{1}}(z_{1})x_{n_2\alpha_{2}}(z_{2})
=(z_{2}-z_{1})^{\text{min}\left\{2n_1,n_2\right\}}x_{n_2\alpha_{2}}(z_{2})x_{n_1\alpha_{1}}(z_{1}).\end{equation*}
\end{proof}

\subsection{Quasi-particle bases}
We start with a definition of a set $B_{W_{N(k\Lambda_{0})}}$:
\begin{equation}\label{eq:S210}
B_{W_{N(k\Lambda_{0})}}= \bigcup_{\substack{0 \leq n_{r_{1}^{(1)},1}\leq \ldots \leq n_{1,1}\\
0\leq n_{r_{2}^{(1)},2}\leq \ldots \leq  n_{1,2}}}\left(\text{or, equivalently} \ \ \bigcup_{\substack{r_{1}^{(1)}\geq
r_{1}^{(2)} \geq \cdots \geq 0\\ r_{2}^{(1)}\geq r_{2}^{(2)}\geq  \cdots  \geq  0}}\right)
\end{equation}
\begin{equation*}\left\{b\right.=b(\alpha_{2})b(\alpha_{1})
=x_{n_{r_{2}^{(1)},2}\alpha_{2}}(m_{r_{2}^{(1)},2})\cdots  x_{n_{1,1}\alpha_{1}}(m_{1,1}):\end{equation*}
\begin{align}\nonumber
\left|
\begin{array}{l}
m_{p,1}\leq  -n_{p,1} - \sum_{p>p'>0} 2 \ \text{min}\{n_{p,1}, 
n_{p',1}\}, \  1\leq p\leq r_{1}^{(1)};\\
m_{p+1,1}\leq  m_{p,1}-2n_{p,1} \  \text{if} \ n_{p,1}=n_{p+1,1}, \  1\leq p \leq  r_{1}^{(1)}-1;\\
m_{p,2}\leq  -n_{p,2}+ \sum_{q=1}^{r_{1}^{(1)}}\text{min}\left\{2n_{q,1},n_{p,2}\right\}- \sum_{p>p'>0} 
2 \ \text{min}\{n_{p,2}, n_{p',2}\},  1\leq  p\leq  r_{2}^{(1)};\\
m_{p+1,2} \leq   m_{p,2}-2n_{p,2} \  \text{if} \ n_{p+1,2}=n_{p,2}, \ 1\leq p \leq  r_{2}^{(1)}-1
\end{array}
\right\}.
\end{align}
We will prove the theorem:
\begin{thm}\label{thm:1} 
The set 
\begin{equation*}
\mathfrak{B}_{W_{N(k\Lambda_{0})}}=\left\{bv_{N(k\Lambda_{0})}: b \in B_{W_{N\left(k\Lambda_{0}\right)}}\right\}
\end{equation*} forms a basis for the principal subspace $W_{N\left(k\Lambda_{0}\right)}$ of $N\left(k\Lambda_{0}\right)$.
\end{thm}
The first step in a proof of the Theorem \ref{thm:1} is the following proposition:
\begin{prop}\label{prop:S21} 
The set $\mathfrak{B}_{W_{N(k\Lambda_{0})}}=\left\{bv_{N(k\Lambda_{0})}: b \in B_{W_{N\left(k\Lambda_{0}\right)}}\right\}$ spans the  principal sub\-space $W_{N\left(k\Lambda_{0}\right)}$ of $N\left(k\Lambda_{0}\right)$.
\end{prop}
\begin{flushright}
$\square$
\end{flushright}
Besides the fact that quasi-particles in our set $B_{W_{N\left(k\Lambda_{0}\right)}}$ have no limit on the charge, the proof of the Proposition \ref{prop:S21} is the same as in \cite{G}, so we omit the proof of Proposition \ref{prop:S21}. 
\par Denote with $f_{k}$ surjective linear map of vector spaces:
\begin{align*}
f_{k}: W_{N(k\Lambda_{0})}\rightarrow W_{L(k\Lambda_{0})}.
\end{align*}
From relations (\ref{al:S15}) and (\ref{al:S16}) follows 
\begin{equation*}f_{k}\left(\mathfrak{B}_{W_{N(k\Lambda_{0})}}\right)=\mathfrak{B}_{W_{L(k\Lambda_{0})}}.\end{equation*}
Elements of the set $\mathfrak{B}_{W_{L(k\Lambda_{0})}}$ are monomial vectors $bv_{k\Lambda_{0}}$, where monomials $b$ are from the set $B_{W_{L\left(k\Lambda_{0}\right)}}$:
\begin{equation}\label{SkupL}
B_{W_{L\left(k\Lambda_{0}\right)}}= \bigcup_{\substack{n_{r_{1}^{(1)},1}\leq \ldots \leq n_{1,1}\leq 
k\\n_{r_{2}^{(1)},2}\leq \ldots \leq n_{1,2}\leq 2k}}\left(\text{or, equivalently} \ \ \ 
\bigcup_{\substack{r_{1}^{(1)}\geq \cdots\geq r_{1}^{(k)}\geq 0\\r_{2}^{(1)}\geq \cdots\geq r_{2}^{(2k)}\geq 
0}}\right)
\end{equation}
\begin{align}\nonumber
\left|
\begin{array}{l}
m_{p,1}\leq  -n_{p,1} - \sum_{p>p'>0} 2 \ \text{min}\{n_{p,1}, 
n_{p',1}\}, \  1\leq  p\leq r_{1}^{(1)};\\
m_{p+1,1}\leq  m_{p,1}-2n_{p,1} \  \text{if} \ n_{p,1}=n_{p+1,1}, \  1\leq  p\leq r_{1}^{(1)}-1;\\
m_{p,2}\leq  -n_{p,2}+ \sum_{q=1}^{r_{1}^{(1)}}\text{min}\left\{2n_{q,1},n_{p,2}\right\}- \sum_{p>p'>0} 
2 \ \text{min}\{n_{p,2}, n_{p',2}\}, \ 1\leq  p\leq r_{2}^{(1)};\\
m_{p+1,2} \leq   m_{p,2}-2n_{p,2} \  \text{if} \ n_{p+1,2}=n_{p,2}, \ 1\leq  p\leq r_{2}^{(1)}-1
\end{array}
\right\}.
\end{align}
Now, from the Proposition \ref{prop:S21} and the above considerations follows: 
\begin{prop}\label{prop:S22} The set $\mathfrak{B}_{W_{L(k{\Lambda}_{0})}}=\left\{bv_{k\Lambda_{0}}:b \in B_{W_{L\left(k\Lambda_{0}\right)}}\right\}$ spans the principal subspace $W_{L\left(k\Lambda_{0}\right)}$.
\end{prop}
\begin{flushright}
$\square$
\end{flushright}

\section{Proof of linear independence}

\subsection{Projection \texorpdfstring{$\pi_{\mathfrak{R}}$}{piR}}\label{ss:proj}
Modeled on a projection from \cite{G}, we shall define a projection $\pi_{\mathfrak{R}}$ distributing quasi-particles of monomial vectors  from the set $\mathfrak{B}_{W_{L(k\Lambda_{0})}}$ among factors of the tensor product of $\mathfrak{h}$-weight subspaces of standard  modules $L(\Lambda_{0})$. If we restrict the action of the Cartan subalgebra $\mathfrak{h}=\mathfrak{h} \otimes 1$ to the principal subspace $W_{L(\Lambda_0)}$ of level $1$ standard  modules $L(\Lambda_{0})$, we get the direct sum of vector spaces: 
 \begin{equation*}
 W_{L(\Lambda_{0})}= \bigoplus_{u,v\geq 0} {W_{L(\Lambda_{0})}}_{(u,v)},
\end{equation*}
where 
\begin{equation*}
{W_{L(\Lambda_{0})}}_{(u,v)}={W_{L(\Lambda_{0})}}_{u\alpha_2 + v\alpha_1}
\end{equation*}
is  the weight subspace of weight 
\begin{equation*}
u\alpha_2 + v\alpha_1 \in Q.
\end{equation*}
\par Fix a level $k> 1$.  We shall write the tensor product of $k$ principal subspaces $W_{L(\Lambda_0)}$ of level $1$ as
\begin{equation*}
W_{L(\Lambda_0)}\otimes \cdots\otimes W_{L(\Lambda_0)}.
\end{equation*}
Like in \cite{G}, we will realize the principal subspace $W_{L(k\Lambda_{0})}$ as a subspace of the tensor product of $k$ principal      subspaces $W_{L(\Lambda_0)}$. That is, 
\begin{align*}
W_{L(k\Lambda_{0})}
\subset  W_{L(\Lambda_0)}\otimes \cdots \otimes  W_{L(\Lambda_0)}
\subset  L(\Lambda_{0})^{\otimes k},
\end{align*}
where
\begin{equation*}
v_{k\Lambda_0}=\underbrace{v_{\Lambda_{0}} \otimes \cdots \otimes v_{\Lambda_{0}}}_{k \ \text{factors}}\end{equation*}
is the highest weight vector of weight $k\Lambda_0$.
\par For a chosen dual-charge-type 
\begin{equation*}
\mathfrak{R}=\left( r_{2}^{(1)}, \ldots ,r_{2}^{(2k)}; r_{1}^{(1)}, \ldots , r_{1}^{(k)}\right),
\end{equation*}
denote with $\pi_{\mathfrak{R}}$ the projection of principal subspace $W_{L(k\Lambda_{0})}$ 
to the subspace
\begin{equation*}
 {W_{L(\Lambda_0)}}_{(\mu^{(k)}_{2};r_{1}^{(k)})}\otimes \cdots \otimes  {W_{L(\Lambda_0)}}_{(\mu^{(1)}_{2};r_{1}^{(1)})},
\end{equation*}
where 
\begin{equation*}  
\mu^{(t)}_{2}=r^{(2t)}_{2}+ r^{(2t-1)}_{2},
\end{equation*}
for every $1 \leq  t \leq k$. We shall denote by the same symbol $\pi_{\mathfrak{R}}$ the generalization of this projection to the space of formal series with coefficients in $W_{L(\Lambda_0)} \otimes \cdots \otimes  W_{L(\Lambda_0)}$. Let
\begin{equation}\label{eq:S31}
x_{n_{r_{2}^{(1)},2}\alpha_{2}}(z_{r_{2}^{(1)},2}) \cdots     x_{n_{1,2}\alpha_{2}}(z_{1,2})x_{n_{r_{1}^{(1)},1}\alpha_{1}}(z_{r_{1}^{(1)},1})\cdots  x_{n_{1,1}\alpha_{1}}(z_{1,1})
\end{equation}
be a generating function of the chosen dual-charge-type $\mathfrak{R}$ and the corresponding charge-type $\mathfrak{R}'$. For the color $i=1$, for every $t$, $1\leq t \leq k$, and every $p$, $1 \leq p \leq r_{1}^{(1)}$, we introduce the symbols 
\begin{equation*}
0 \leq  n^{(t)}_{p,1} \leq 1,
\end{equation*}
where 
\begin{equation*}
n^{(1)}_{p,1}\geq n^{(2)}_{p,1}\geq \ldots \geq  n^{(k-1)}_{p,1}\geq  n^{(k)}_{p,1}.
\end{equation*}
and
\begin{equation*} 
n_{p,1}=n^{(1)}_{p,1}+n^{(2)}_{p,1}+ \cdots +n^{(k-1)}_{p,1}+n^{(k)}_{p,1}.
\end{equation*}
Similarly, for color $i=2$ we introduce the symbols
\begin{equation*}
0 \leq  n^{(t)}_{p,2}\leq  2,
\end{equation*}
for every $t$, $1 \leq t \leq k$ and every $p$, $1 \leq p \leq r_{2}^{(1)}$, so that at most one $n^{(t)}_{p,2}$ can be $1$,
\begin{equation*}
n^{(1)}_{p,2}\geq  n^{(2)}_{p,2} \geq  \ldots  \geq  n^{(k-1)}_{p,2} \geq  n^{(k)}_{p,2}
\end{equation*}
and
\begin{equation*}
n_{p,2}=n^{(1)}_{p,2}+n^{(2)}_{p,2}+ \cdots +n^{(k-1)}_{p,2}+n^{(k)}_{p,2}.
\end{equation*}
From relations (\ref{al:S15}), (\ref{al:S16}) and definition of the action of Lie algebra on the modules, follows that the projection of the generating function (\ref{eq:S31}) is
\begin{align}
\label{projekcija}
\pi_{\mathfrak{R}}& x_{n_{r_{2}^{(1)},2}\alpha_{2}}(z_{r_{2}^{(1)},2})\cdots  x_{n_{1,1}\alpha_{1}}(z_{1,1}) \ v_{k\Lambda_{0}}&\\
\nonumber
=&\text{C} \ x_{n^{(k)}_{r^{(2k-1)}_{2},2}\alpha_{2}}(z_{r_{2}^{(2k-1)},2})\cdots                x_{n^{(k)}_{r^{(2k)}_{2},2}\alpha_{2}}(z_{r^{(2k)}_{2},2})\cdots  x_{n^{(k)}_{1,2}\alpha_{2}}(z_{1,2})&\\
\nonumber
& \ \ \ \ \ \ \ \ \ \ \ \ \ \ \ \ \ \ \ \ \ \ \ \ \ \ \ \ \ \ \ \ \ \ \ \  x_{n^{(k)}_{r^{(k)}_{1},1}\alpha_{1}}(z_{r_{1}^{(k)},1})\cdots   x_{n^{(k)}_{1,1}\alpha_{1}}(z_{1,1}) \ v_{\Lambda_{0}}&\\
\nonumber
& \ \ \ \ \ \ \ \ \ \ \ \ \ \ \ \ \otimes \ldots \otimes&\\
\nonumber
\otimes &x_{n_{r^{(1)}_{2},2}^{(1)}\alpha_{2}}(z_{r_{2}^{(1)},2})\cdots  x_{n_{r^{(2)}_{2},2}^{(1)}\alpha_{2}}(z_{r_{2}^{(2)},2})\cdots 
  x_{n_{1,2}^{(1)}\alpha_{2}}(z_{1,2})& \\ 
\nonumber
& \ \ \ \ \ \ \ \ \ \ \ \ \ \ \ \ \ \ \ \ \ \ \ \ \ \ \ \ \ \ \ \ \ \ \ \ x_{n_{r^{(1)}_{1},1}^{(1)}\alpha_{1}}(z_{r_{1}^{(1)},1})\cdots     x_{n{_{1,1}^{(1)}\alpha_{1}}}(z_{1,1}) \ v_{\Lambda_{0}},&
\end{align}
where $\text{C} \in \mathbb{C}^{*}$. In the projection (\ref{projekcija}), $n_{p,1}$ generating functions $x_{\alpha_{1}}(z_{p,1})$ ($1\leq p \leq r^{(1)}_{1}$), whose product generates a quasi-particle of charge $n_{p,1}$, 
 ``are placed at'' the first (from right to left) $n_{p,1}$ tensor factors: 
\begin{equation*}
 x_{n^{(k)}_{p,1}\alpha_{1}}(z_{p,1})\otimes  x_{n^{(k-1)}_{p,1}\alpha_{1}}(z_{p,1}) \otimes \cdots \otimes      x_{n^{(2)}_{p,1}\alpha_{1}}(z_{p,1})\otimes  x_{n^{(1)}_{p,1}\alpha_{1}}(z_{p,1}),
\end{equation*}
so that, in the $t$-tensor row ($1 \leq t \leq k$), we have: 
\begin{equation*}
 \cdots x_{n_{r^{(t)}_{1},1}^{(t)}\alpha_{1}}(z_{r_{1}^{(t)},1})
  \cdots  x_{n{_{1,1}^{(t)}\alpha_{1}}}(z_{1,1})v_{\Lambda_{0}} \otimes \cdots \ \ .
\end{equation*}
Figuratively, this can be shown as in the example in Figure 2, where each box represents $ n_ {p,1}^{(t)}$. 
\begin{figure}[h!tb] 
\centering
\setlength{\unitlength}{5mm}
\begin{picture}(10,10)
\linethickness{0.075mm}
\multiput(0,0)(1,0){1}
{\line(0,1){1}}
\multiput(1,0)(1,0){1}
{\line(0,1){2}}
\multiput(2,0)(1,0){5}
{\line(0,1){3}}
\multiput(2,1)(1,0){1}
{\line(1,0){6}}
\multiput(2,3)(1,0){1}
{\line(1,0){7}}
\multiput(1,0)(1,0){3}
{\line(0,1){2}}
\multiput(1,2)(0,1){1}
{\line(13,0){8}}
\multiput(4,0)(1,0){1}
{\line(0,1){2}}
\multiput(5,0)(1,0){2}
{\line(0,1){2}}
\multiput(8,0)(1,0){1}
{\line(0,1){2}}
\multiput(9,0)(1,0){1}
{\line(0,1){3}}
\multiput(8,5)(1,0){2}
{\line(0,1){3}}
\multiput(7,0)(1,0){1}
{\line(0,1){2}}
\multiput(7,5)(1,0){2}
{\line(0,1){3}}
\multiput(5,5)(1,0){2}
{\line(0,1){3}}
\multiput(5,7)(1,0){2}
{\line(1,0){3}}
\multiput(5,8)(1,0){2}
{\line(1,0){3}}
\multiput(4,7)(1,0){2}
{\line(1,0){3}}
\multiput(4,6)(1,0){2}
{\line(1,0){4}}
\multiput(4,5)(1,0){2}
{\line(1,0){4}}
\multiput(4,5)(1,0){1}
{\line(0,1){2}}
\multiput(0,0)(0,1){2}
{\line(1,0){9}}
\multiput(5,0)(1,0){2}
{\line(0,1){2}}
\multiput(7,0)(1,0){1}
{\line(0,1){3}}
\multiput(8,0)(1,0){2}
{\line(0,1){3}}
\multiput(6,0)(1,0){2}
{\line(0,1){2}}
\put(0,-0.5){\scriptsize{\footnotesize{$n_{r_1^{(1)},1}$}}}
\put(5.1,-0.5){\scriptsize{\footnotesize{$n_{r_1^{(k)},1}$}}}
\put(8.2,-0.5){\scriptsize{\footnotesize{$n_{1,1}$}}}
\put(4.5,3.7){$\textbf{\vdots}$}
\put(5.5,3.7){$\textbf{\vdots}$}
\put(6.5,3.7){$\textbf{\vdots}$}
\put(7.5,3.7){$\textbf{\vdots}$}
\put(8.5,3.7){$\textbf{\vdots}$}
\put(9.3,3.7){$\textbf{\vdots}$}
\put(-0.9,3.7){$\textbf{\vdots}$}
\put(-1.1,0.3){\footnotesize{$r_1^{(1)}$}}
\put(-1.1,1.3){\footnotesize{$r_1^{(2)}$}}
\put(-1.1,2.3){\footnotesize{$r_1^{(3)}$}}
\put(9.3,0.3){\footnotesize{$v_{\Lambda_0}$}}
\put(9.3,1.3){\footnotesize{$v_{\Lambda_0}$}}
\put(9.3,2.3){\footnotesize{$v_{\Lambda_0}$}}
\put(9.3,5.3){\footnotesize{$v_{\Lambda_0}$}}
\put(9.3,6.3){\footnotesize{$v_{\Lambda_0}$}}
\put(9.3,7.3){\footnotesize{$v_{\Lambda_0}$}}
\put(-1.1,5.3){\footnotesize{$r_1^{(k-2)}$}}
\put(-1.1,6.3){\footnotesize{$r_1^{(k-1)}$}}
\put(-1.1,7.3){\footnotesize{$r_1^{(k)}$}}
\end{picture}
\bigskip
\caption{Sketch of projection $\pi_{\mathfrak{R}}$ for color $i=1$}
\end{figure}
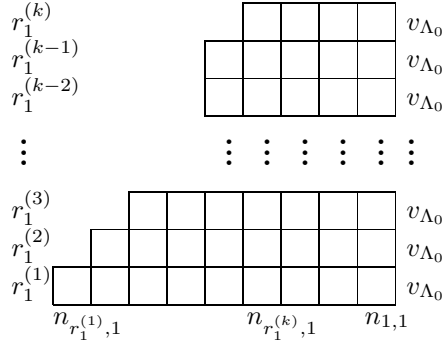
\par From the relation $x_{3\alpha_{2}}(z)=0$ on the principal subspace $W_{L(\Lambda_0)}$, follows that, with the projection   $\pi_{\mathfrak{R}}$ at most, two generating functions of color $i=2$ ``can be placed at'' every tensor factor. 
 If $n_{p,2}$ ($1\leq p \leq r_2^{(1)}$) is an even number, then two generating functions $x_{\alpha_{2}}(z_{p,2})$ ``are placed at'' the  first $\frac{n_{p,2}}{2}$ tensor factors (from right to left). If $n_{p,2}$ is an odd number, then two generating functions  $x_{\alpha_{2}}(z_{p,2})$ ``are placed at'' the first $\frac{n_{p,2}-1}{2}$ tensor factors (from right to left), and the last generating  function $x_{\alpha_{2}}(z_{p,2})$ ``is placed at'' $\frac{n_{p,2}-1}{2}+1$ tensor factor:  
\begin{equation*}
 x_{n^{(k)}_{p,2}\alpha_{1}}(z_{p,2}) \otimes  x_{n^{(k-1)}_{p,2}\alpha_{1}}(z_{p,2}) \otimes  \cdots \otimes   x_{n^{(2)}_{p,2}\alpha_{1}}(z_{p,2})\otimes  x_{n^{(1)}_{p,2}\alpha_{1}}(z_{p,2}),
\end{equation*} 
so that, in every $t$-tensor row ($1 \leq t \leq k$), we have: 
\begin{equation*}
\cdots \otimes x_{n_{r^{(2t-1)}_{2},2}^{(t)}\alpha_{2}}(z_{r_{1}^{(2t-1)},2}) \cdots        x_{n_{r^{(2t)}_{2},2}^{(t)}\alpha_{2}}(z_{r_{2}^{(2t)},2})
\cdots  x_{n{_{1,2}^{(t)}\alpha_{2}}}(z_{1,2})\cdots  v_{\Lambda_{0}}\otimes \cdots \ \ . 
\end{equation*}
As in the previous case, this situation can be shown as in the example in Figure 3.
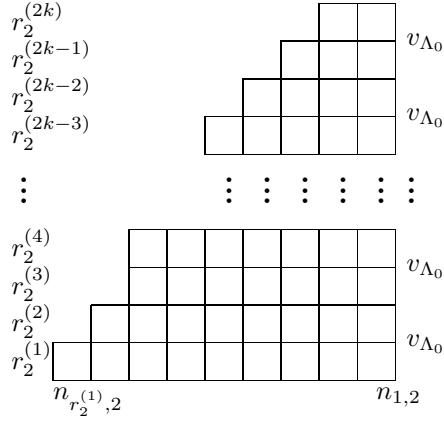
\begin{figure}[h!tb]
\centering
\setlength{\unitlength}{5mm}
\begin{picture}(10,10)
\linethickness{0.075mm}
\multiput(0,0)(1,0){1}
{\line(0,1){1}}
\multiput(1,0)(1,0){1}
{\line(0,1){2}}
\multiput(2,0)(1,0){8}
{\line(0,1){4}}
\multiput(2,1)(1,0){1}
{\line(1,0){6}}
\multiput(2,3)(1,0){1}
{\line(1,0){7}}
\multiput(2,4)(1,0){1}
{\line(1,0){7}}
\multiput(1,0)(1,0){3}
{\line(0,1){2}}
\multiput(1,2)(0,1){1}
{\line(13,0){8}}
\multiput(4,0)(1,0){1}
{\line(0,1){2}}
\multiput(5,0)(1,0){2}
{\line(0,1){2}}
\multiput(8,0)(1,0){1}
{\line(0,1){2}}
\multiput(9,0)(1,0){1}
{\line(0,1){3}}
\multiput(8,6)(1,0){2}
{\line(0,1){3}}
\multiput(7,0)(1,0){1}
{\line(0,1){2}}
\multiput(7,6)(1,0){2}
{\line(0,1){3}}
\multiput(5,6)(1,0){2}
{\line(0,1){2}}
\multiput(5,7)(1,0){2}
{\line(1,0){3}}
\multiput(5,8)(1,0){2}
{\line(1,0){3}}
\multiput(6,9)(1,0){2}
{\line(1,0){2}}
\multiput(6,8)(1,0){1}
{\line(0,1){1}}
\multiput(7,8)(1,0){3}
{\line(0,1){2}}
\multiput(7,10)(1,0){2}
{\line(1,0){1}}
\multiput(5,8)(1,0){2}
{\line(1,0){2}}
\multiput(4,7)(1,0){2}
{\line(1,0){4}}
\multiput(4,6)(1,0){2}
{\line(1,0){4}}
\multiput(4,6)(1,0){1}
{\line(0,1){1}}
\multiput(0,0)(0,1){2}
{\line(1,0){9}}
\multiput(5,0)(1,0){2}
{\line(0,1){2}}
\multiput(7,0)(1,0){1}
{\line(0,1){3}}
\multiput(8,0)(1,0){2}
{\line(0,1){3}}
\multiput(6,0)(1,0){2}
{\line(0,1){2}}
\put(0,-0.5){\footnotesize{$n_{r_2^{(1)},2}$}}
\put(8.5,-0.5){\scriptsize{\footnotesize{$n_{1,2}$}}}
\put(4.5,4.7){$\textbf{\vdots}$}
\put(5.5,4.7){$\textbf{\vdots}$}
\put(6.5,4.7){$\textbf{\vdots}$}
\put(7.5,4.7){$\textbf{\vdots}$}
\put(8.5,4.7){$\textbf{\vdots}$}
\put(9.3,4.7){$\textbf{\vdots}$}
\put(-0.9,4.7){$\textbf{\vdots}$}
\put(-1.1,0.3){\footnotesize{$r_2^{(1)}$}}
\put(-1.1,1.3){\footnotesize{$r_2^{(2)}$}}
\put(-1.1,2.3){\footnotesize{$r_2^{(3)}$}}
\put(-1.1,3.3){\footnotesize{$r_2^{(4)}$}}
\put(9.3,0.9){\footnotesize{$v_{\Lambda_0}$}}
\put(9.3,2.9){\footnotesize{$v_{\Lambda_0}$}}
\put(9.3,6.9){\footnotesize{$v_{\Lambda_0}$}}
\put(9.3,8.9){\footnotesize{$v_{\Lambda_0}$}}
\put(-1.1,6.3){\footnotesize{$r_2^{(2k-3)}$}}
\put(-1.1,7.3){\footnotesize{$r_2^{(2k-2)}$}}
\put(-1.1,8.3){\footnotesize{$r_2^{(2k-1)}$}}
\put(-1.1,9.3){\footnotesize{$r_2^{(2k)}$}}
\end{picture}
\bigskip
\caption{Sketch of projection $\pi_{\mathfrak{R}}$ for color $i=2$}
\end{figure}
\par Let $b\in B_{W_{L(k\Lambda_0)}}$ be a monomial 
\begin{equation}\label{eq:pol}
b=x_{n_{r_{2}^{(1)},2}\alpha_{2}}(m_{r_{2}^{(1)},2})\cdots      x_{n_{1,2}\alpha_{2}}(m_{1,2})x_{n_{r_{1}^{(1)},1}\alpha_{1}}(m_{r_{1}^{(1)},1})\cdots  x_{n_{1,1}\alpha_{1}}(m_{1,1})
\end{equation}
colored with color-type $(r_{2},r_{1}),$ charge-type $\mathfrak{R}'$ and dual-charge-type $\mathfrak{R}$. It follows from the above considerations that the projection of monomial vector $bv_{k\Lambda_{0}}$ is a coefficient of the projection of the generating function  (\ref{projekcija}) which we denote as $$\pi_{\mathfrak{R}}bv_{k\Lambda_{0}}.$$ 
 The following lemma will be used in the proof of the main theorem of this section. The proof of lemma follows from the definition of the  projection and relations (\ref{al:S15}) and (\ref{al:S16}) for level one modules:
\begin{lem}\label{S62L}
Let $b, \bar{b}\in B_{W_{L(k\Lambda_0)}}$ be monomials such that monomial $b$ is charge-type $\mathfrak{R}'$ and dual-charge-type $\mathfrak{R}$ and monomial $\bar{b}$ is charge-type $$(\bar{n}_{{\bar{r}}_{2}^{(1)},2}, \ldots ,  \bar{n}_{1,2};\bar{n}_{{\bar{r}}_{1}^{(1)},1}, \ldots , \bar{n}_{1,1})$$ and corresponding dual-charge-type 
$\bar{\mathfrak{R}}=\left( {\bar{r}}_{2}^{(1)}, \ldots ,{\bar{r}}_{2}^{(2k)}; {\bar{r}}_{1}^{(1)}, \ldots , {\bar{r}}_{1}^{(k)}\right),$ so that 
$b<\bar{b}.$
 Then we have $\pi_{\mathfrak{R}}\bar{b}v_{k\Lambda_{0}}=0$.
\end{lem}
\begin{flushright}
$\square$
\end{flushright}
\par We conduct the proof of linear independence of the set $\mathfrak{B}_{W_{L(k\Lambda_{0})}}$ in a few steps. In the final step, we will use the linear independence of the set that spans the principal subspace of standard module of affine Lie algebra of type $\widetilde{sl}_2(\alpha_2)$. Therefore, we especially analyze the projection of monomial vectors $b(\alpha_{2})v_{k\Lambda_0}$ from the set $\mathfrak{B}_{W_{L(k\Lambda_{0})}}$.  Denote by
\begin{equation*}
\mathfrak{A} \subset \mathfrak{B}_{W_{L(k\Lambda_{0})}}
\end{equation*}
the set of monomial vectors $bv_{k\Lambda_0}$, where every $b$ is a monomial colored with the color $i=2$. Let 
\begin{equation*}
b(\alpha_{2})v_{k\Lambda_0}
\end{equation*} be an element of the set $\mathfrak{A}$ such that $b(\alpha_{2})$: 
\begin{equation*}
b(\alpha_{2})=x_{n_{r_{2}^{(1)},2}\alpha_{2}}(m_{r_{2}^{(1)},2})\cdots  x_{n_{1,2}\alpha_{2}}(m_{1,2}), 
\end{equation*} 
is of the charge-type
\begin{equation*} 
\left(n_{r_{2}^{(1)},2}, \ldots, n_{1,2}\right) ,
\end{equation*}
\begin{equation*} 
n_{r_{2}^{(1)},2}\leq \ldots \leq n_{1,2}\leq 2k 
\end{equation*}
and of the dual-charge-type 
\begin{equation*}
\mathfrak{R}=\left(r_{2}^{(1)}, \ldots, r_{2}^{(2k)}\right) ,
\end{equation*}
\begin{equation*}
r_{2}^{(1)}\geq \ldots \geq r_{2}^{(2k)},
\end{equation*}
with the satisfied requirements for the $m_{p,2}$ ($1\leq p\leq r_{2}^{(1)}$):
\begin{align*}
m_{1,2}\leq & -n_{1,2}&\\
m_{p,2}\leq & -n_{p,2}- \sum_{p>p'>0}\ 2 \ \text{min}\{n_{p,2}, n_{p',2}\} \ \ \ \text{for} \ \ 2 \leq  p \leq  r_{2}^{(1)}&\\
m_{p+1,2}\leq & m_{p,2}-2n_{p,2}, \ \ \ \text{if} \ \ \ n_{p+1,2}=n_{p,2}, \ \ \  1\leq p \leq r_{2}^{(1)}-1.&
\end{align*}
Denote by
\begin{equation*}
x_{n_{r_{2}^{(1)},2}\alpha_{2}}(z_{r_{2}^{(1)},2})\cdots  x_{n_{1,2}\alpha_{2}}(z_{1,2})
\end{equation*}
the generating function of the same charge-type and dual-charge-type as the monomial $b(\alpha_{2})$. Under consideration at the beginning  of this section, the default dual-charge-type $\mathfrak{R}$ determines the projection  $\pi_{\mathfrak{R}}$ on the vector space 
\begin{equation*}
{W_{L(\Lambda_{0})}}_{(\mu^{(k)}_{2};0)}\otimes \cdots \otimes {W_{L(\Lambda_{0})}}_{(\mu^{(1)}_{2};0)} \subset W_{L(\Lambda_{0})}\otimes \cdots \otimes  W_{L(\Lambda_{0})}.
\end{equation*}
\par Since the restriction of $B^{(1)}_{2}$-module $L(\Lambda_{0})$ on the subalgebra $\widetilde{sl}_{2}(\alpha_{2})$ is a direct sum of  standard  $\widetilde{sl}_{2}(\alpha_{2})$-modules of level $2$,  we will introduce the following symbol $L^B(\Lambda_{0})$ for standard   $B_2^{(1)}$-module with the highest weight vector $v^B_{\Lambda_0}$ and the mark $L^A(2\Lambda_{0})\subset L^B(\Lambda_{0})$ for standard   $A_1^{(1)}$-module of level $2$ with the highest weight vector $v^A_{2\Lambda_0}=v^B_{\Lambda_0}$. From the condition
 $x_{3\alpha_2}(z)=0$, it follows that $\pi_{\mathfrak{R}}bv_{k\Lambda_{0}}$ is an element of vector space 
\begin{equation*}
\underbrace{W_{L^A(2\Lambda_{0})}\otimes  \ldots  \otimes  W_{L^A(2\Lambda_{0})}}_{k \ \text{factors}},
\end{equation*}
where $W_{L^A(2\Lambda_{0})}={W_{L^B(\Lambda_{0})}}_{0 \cdot  \alpha_1}$ is the principal subspace of standard module $L^A(2\Lambda_{0})$   of the affine Lie algebra $\widetilde{sl}_{2}(\alpha_{2})$. Denote by 
\begin{align*}
{W_{L^A\left(2\Lambda_{0}\right)}}_{(u)}= {W_{L^B(\Lambda_{0})}}_{u\alpha_2},
\end{align*}
$\mathfrak{h}$-weighted subspace of $W_{L^A(2\Lambda_{0})}$.   
\begin{figure}
\centering
\setlength{\unitlength}{5mm}
\begin{picture}(10,10)
\linethickness{0.075mm}
\multiput(0,0)(1,0){1}
{\line(0,1){1}}
\multiput(1,0)(1,0){1}
{\line(0,1){2}}
\multiput(2,0)(1,0){8}
{\line(0,1){4}}
\multiput(2,1)(1,0){1}
{\line(1,0){6}}
\multiput(2,3)(1,0){1}
{\line(1,0){7}}
\multiput(2,4)(1,0){1}
{\line(1,0){7}}
\multiput(1,0)(1,0){3}
{\line(0,1){2}}
\multiput(1,2)(0,1){1}
{\line(13,0){8}}
\multiput(4,0)(1,0){1}
{\line(0,1){2}}
\multiput(5,0)(1,0){2}
{\line(0,1){2}}
\multiput(8,0)(1,0){1}
{\line(0,1){2}}
\multiput(9,0)(1,0){1}
{\line(0,1){3}}
\multiput(8,6)(1,0){2}
{\line(0,1){3}}
\multiput(7,0)(1,0){1}
{\line(0,1){2}}
\multiput(7,6)(1,0){2}
{\line(0,1){3}}
\multiput(5,6)(1,0){2}
{\line(0,1){2}}
\multiput(5,7)(1,0){2}
{\line(1,0){3}}
\multiput(5,8)(1,0){2}
{\line(1,0){3}}
\multiput(6,9)(1,0){2}
{\line(1,0){2}}
\multiput(6,8)(1,0){1}
{\line(0,1){1}}
\multiput(7,8)(1,0){3}
{\line(0,1){2}}
\multiput(7,10)(1,0){2}
{\line(1,0){1}}
\multiput(5,8)(1,0){2}
{\line(1,0){2}}
\multiput(4,7)(1,0){2}
{\line(1,0){4}}
\multiput(4,6)(1,0){2}
{\line(1,0){4}}
\multiput(4,6)(1,0){1}
{\line(0,1){1}}
\multiput(0,0)(0,1){2}
{\line(1,0){9}}
\multiput(5,0)(1,0){2}
{\line(0,1){2}}
\multiput(7,0)(1,0){1}
{\line(0,1){3}}
\multiput(8,0)(1,0){2}
{\line(0,1){3}}
\multiput(6,0)(1,0){2}
{\line(0,1){2}}
\put(0,-0.5){\footnotesize{$n_{r_2^{(1)},2}$}}
\put(8.2,-0.5){\scriptsize{\footnotesize{$n_{1,2}$}}}
\put(4.5,4.7){$\textbf{\vdots}$}
\put(5.5,4.7){$\textbf{\vdots}$}
\put(6.5,4.7){$\textbf{\vdots}$}
\put(7.5,4.7){$\textbf{\vdots}$}
\put(8.5,4.7){$\textbf{\vdots}$}
\put(9.3,4.7){$\textbf{\vdots}$}
\put(-0.9,4.7){$\textbf{\vdots}$}
\put(-0.9,0.9){\footnotesize{$\mu^{(1)}$}}
\put(-0.9,2.9){\footnotesize{$\mu^{(2)}$}}
\put(9.3,0.9){\footnotesize{$v^A_{2\Lambda_0}$}}
\put(9.3,2.9){\footnotesize{$v^A_{2\Lambda_0}$}}
\put(9.3,6.9){\footnotesize{$v^A_{2\Lambda_0}$}}
\put(9.3,8.9){\footnotesize{$v^A_{2\Lambda_0}$}}
\put(-0.8,6.9){\footnotesize{$\mu^{(k-1)}$}}
\put(-0.9,8.9){\footnotesize{$\mu^{(k)}$}}
\end{picture}
\bigskip
\caption{Sketch of projection of $b(\alpha_2)v_{k\Lambda_0}$}
\end{figure}
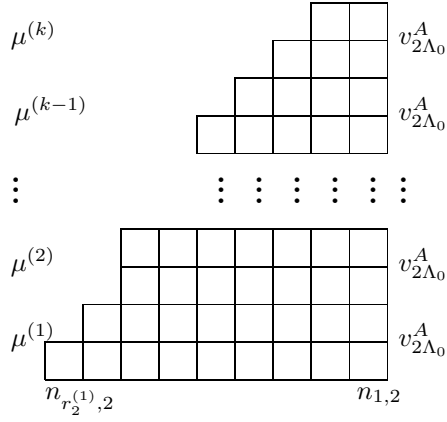
 On every factor in the tensor product $W_{L^A(2\Lambda_{0})}\otimes \cdots \otimes W_{L^A(2\Lambda_{0})}$ of $k$ principal subspaces       $W_{L^A(2\Lambda_{0})}$, we have embedding 
\begin{equation*}
W_{L^A(2\Lambda_{0})} \hookrightarrow W_{L^A(\Lambda_{0})} \otimes  W_{L^A(\Lambda_{0})},
\end{equation*}
that is  
\begin{equation*}
{W_{L^A(2\Lambda_{0})}}_{(\mu^{(p)}_{2})}\hookrightarrow \sum_{\stackrel{u,v\in 
\mathbb{N}}{\mu^{(p)}_{2}=u+v}}{W_{L^A(\Lambda_{0})}}_{(u)}\otimes  {W_{L^A(\Lambda_{0})}}_{(v)},
\end{equation*}
 for $1\leq p \leq k$. 
\par Denote by $\pi'_{\mathfrak{R}}$ the projection of the vector space
\begin{equation*}
W_{L^A(2\Lambda_{0})}\otimes  \cdots  \otimes  W_{L^A(2\Lambda_{0})}\end{equation*}
 on subspace
\begin{equation*}
{W_{L^A(\Lambda_{0})}}_{(r_{2}^{(2k)})}\otimes  {W_{L^A(\Lambda_{0})}}_{(r_{2}^{(2k-1)})}\otimes  \cdots  \otimes       {W_{L^A(\Lambda_{0})}}_{(r_{2}^{(2)})}\otimes  {W_{L^A(\Lambda_{0})}}_{(r_{2}^{(1)})}.
\end{equation*}
In particular, extending the above projection on the space of formal series with coefficients in
\begin{equation*}
\underbrace{W_{L^A(\Lambda_{0})}\otimes \cdots \otimes  W_{L^A(\Lambda_{0})}}_{2k \ \text{factors}}\end{equation*}
 from the condition $x_{\alpha_2}(z)=0$ follows
\begin{align*}
&\pi'_{\mathfrak{R}}\left(\pi_{\mathfrak{R}}b(\alpha_2)v_{k\Lambda_{0}}\right)&\\
 \in & {W_{L^A(\Lambda_{0})}}_{(r_{2}^{(2k)})}\otimes {W_{L^A(\Lambda_{0})}}_{(r_{2}^{(2k-1)})}\otimes \cdots \otimes   {W_{L^A(\Lambda_{0})}}_{(r_{2}^{(2)})}\otimes {W_{L^A(\Lambda_{0})}}_{(r_{2}^{(1)})}.&
\end{align*}
Georgiev showed that 
\begin{equation*}
\pi'_{\mathfrak{R}}\circ {\pi_{\mathfrak{R}}}_{\left|_{W_{L^A(2k\Lambda_{0})}}\right.}\mathfrak{A}_{\mathfrak{R}}\end{equation*} is a  linearly independent set. Thus, it follows that the set $\pi_{\mathfrak{R}}\mathfrak{A}_{\mathfrak{R}}$ is a linearly independent set.

\subsection{A coefficient of an intertwining operator}\label{ss:intert}
Here we introduce operators which we use in our proof of linear independence of the set $\mathfrak{B}_{W_{L(k\Lambda_{0})}}$. 
\par Denote by $Y (\cdot, z)$ the vertex operator which determines the structure of $L(\Lambda_0)$-module $L(\Lambda_1)$. From the definition, of the intertwining operator it follows that $Y$ is the intertwining operator of type 
\begin{equation*} 
\binom{L(\Lambda_1)}{L(\Lambda_0) \ L(\Lambda_1)},
\end{equation*}
and $I(\cdot , z)$ defined by  
\begin{align}\label{al:S33}
I(w,z)v=\exp(zL(-1))Y(v,-z)w, \ \  w \in L(\Lambda_1), v \in L(\Lambda_0)
\end{align}
is an intertwining operator of type
\begin{equation*} 
\binom{L(\Lambda_1)}{L(\Lambda_1) \ L(\Lambda_0)},
\end{equation*}
(cf. \cite{FHL}). 
In the next lemma we use the following commutator formula 
\begin{equation*} 
\left[x(m),I(v_{\Lambda_1},z)\right]=\sum_{j\geq 0} \binom {m}{j}z^{m-j}I(x(j)v_{\Lambda_1},z)
\end{equation*}
(cf. formula (2.13) of \cite{Li2}), where $x(m) \in \widehat{\mathfrak{g}}$.  
More precisely:
\begin{lem}\label{lem:S321}
For fixed $m \in \mathbb{Z}$, we have:
\begin{equation*}
\left[x_{\alpha_i}(m), I(v_{\Lambda_1},z)\right]=0.
\end{equation*}
\begin{flushright}
\hfill
\end{flushright}
\end{lem}
\par We define the following coefficient of intertwining operator 
\begin{equation*} 
A_{\omega_1}=\text{Res}_z \ z^{-1}I(v_{\Lambda_1}, z).
\end{equation*}
From definition (\ref{al:S33}), we have
\begin{equation}\label{eq:S34}
A_{\omega_1}v_{\Lambda_0}=v_{\Lambda_1}.
\end{equation}
Set
\begin{equation}\label{eq:S35}
1\otimes\cdots \otimes  A_{\omega_{1}} \otimes \underbrace{1 \otimes \cdots \otimes 1}_{s-1 \ \text{factors}},
\end{equation}
where $s\leq k$. 
Now, if we act with the operator (\ref{eq:S35}) on the vector $v_{k\Lambda_0}$ from (\ref{eq:S34}), it follows:
\begin{equation}\label{eq:S36}
\left(1\otimes\cdots \otimes 1\otimes A_{\omega_{1}} \otimes 1 \otimes \cdots \otimes 1\right)(v_{k\Lambda_0})=v_{\Lambda_{0}}\otimes
\cdots \otimes v_{\Lambda_{0}}\otimes v_{\Lambda_{1}}\otimes \underbrace{v_{\Lambda_{0}}\otimes \cdots\otimes v_{\Lambda_{0}}}_{s-1 \ \text{factors}}.
\end{equation}
\par Set $b\in B_{W_{L(k\Lambda_0)}}$ as in (\ref{eq:pol}). From the consideration in section \ref{ss:proj}, it follows that \begin{equation*}(1\otimes \cdots  \otimes 1 \otimes  A_{\omega_{1}} \otimes 1 \otimes \cdots \otimes 1)\pi_{\mathfrak{R}}bv_{k\Lambda_0}
\end{equation*}
is the coefficient of  
\begin{equation*}
(1\otimes\cdots \otimes A_{\omega_{1}} \otimes 1 \otimes \cdots \otimes  1)\pi_{\textsl{\emph{R}}}x_{n_{r_{2}^{(1)},2}\alpha_{2}}(z_{r_{2}^{(1)},2})\cdots x_{s\alpha_{1}}(z_{1,1})v_{k\Lambda_{0}}.
\end{equation*}
From (\ref{eq:S36}), it follows that operator $A_{\omega_{1}}$ acts only on the $s$-th tensor row: 
\begin{equation*}\otimes  x_{n^{(s)}_{r^{(2s-1)}_{2},2}\alpha_{2}}(z_{r_{2}^{(2s-1)},2})\cdots   x_{n^{(s)}_{r^{(2s)}_{2},2}\alpha_{2}}(z_{r^{(2s)}_{2},2})\cdots  x_{n^{(s)}_{1,2}\alpha_{2}}(z_{1,2})\end{equation*}
\begin{equation*}x_{n^{(s)}_{r^{(s)}_{1},1}\alpha_{1}}(z_{r_{1}^{(s)},1})\cdots  x_{\alpha_{1}}(z_{1,1})v_{\Lambda_{0}}\otimes,\end{equation*}
where $ 0 \leq  n^{(s)}_{p,1} \leq 1$, for $1\leq  p \leq  r^{(s)}_{1}$ and $ 0 \leq  n^{(s)}_{p,2} \leq 2$, for $1\leq p \leq    r^{(2s-1)}_{2}$ (see (\ref{projekcija})). Since $A_{\omega_{1}}$ commutes with the generating functions, in the $s$-th tensor row, we have
\begin{equation*} \otimes x_{n^{(s)}_{r^{(2s-1)}_{2},2}\alpha_{2}}(z_{r_{2}^{(2s-1)},2})\cdots     x_{n^{(s)}_{r^{(2s)}_{2},2}\alpha_{2}}(z_{r^{(2s)}_{2},2})\cdots  x_{n^{(s)}_{1,2}\alpha_{2}}(z_{1,2})\end{equation*}
\begin{equation}\label{Tanja}x_{n^{(s)}_{r^{(s)}_{1},1}\alpha_{1}}(z_{r_{1}^{(s)},1})\cdots  x_{\alpha_{1}}(z_{1,1})v_{\Lambda_{1}}\otimes.
 \end{equation}

\subsection{Simple current operator \texorpdfstring{$e_{\omega_{1}}$}{eomega1}}\label{ss:curr}
Fix $\omega_1 \in \mathfrak{h}$ as in section \ref{ss:modules}. Here we introduce the simple current operator $e_{\omega_{1}}$ (cf. \cite{P}). We use this map  in the proof of linear independence of the set $\mathfrak{B}_{W_{L(k\Lambda_{0})}}$ in the same way as it is used in \cite{G}.
\par Simple current operator $e_{\omega_{1}}$ is a linear bijection between the level $1$ standard modules
\begin{equation*}
e_{\omega_{1}}:L(\Lambda_0)\rightarrow L(\Lambda_1),
\end{equation*}
such that
\begin{equation}\label{S341}
x_{\alpha}(z)e_{\omega_{1}}=e_{\omega_{1}}z^{\alpha(\omega_{1})}x_{\alpha}(z),
\end{equation}
for all $\alpha \in R$, or, written by components,
\begin{equation}\label{S342}
x_{\alpha}(m)e_{\omega_{1}}=e_{\omega_{1}}x_{\alpha}(m+\alpha(\omega_{1})),
\end{equation}
for all $\alpha \in R$ and $m \in \mathbb{Z}$ (cf. \cite{DLM}, \cite{P}).
\par Primc in \cite{P} showed that 
\begin{equation*}
e_{\omega_{1}}{v}_{\Lambda_{0}}=v_{\Lambda_{1}}.
\end{equation*}
Thus, we can rewrite (\ref{Tanja}) as 
\begin{align}\label{glupaTanja}
\cdots \otimes x_{n^{(s)}_{r^{(2s-1)}_{2},2}\alpha_{2}}(z_{r_{2}^{(2s-1)},2})\cdots     x_{n^{(s)}_{r^{(2s)}_{2},2}\alpha_{2}}(z_{r^{(2s)}_{2},2})\cdots  x_{n^{(s)}_{1,2}\alpha_{2}}(z_{1,2})\\
 \nonumber
x_{n^{(s)}_{r^{(s)}_{1},1}\alpha_{1}}(z_{r_{1}^{(s)},1})\cdots  x_{\alpha_{1}}(z_{1,1})e_{\omega_1}v_{\Lambda_{0}}\otimes \cdots. 
 \end{align}
\par We define the linear bijection
\begin{equation}\label{S657}
1\otimes\ldots \otimes 1\otimes e_{\omega_{1}} \otimes \underbrace{1 \otimes \ldots \otimes 1}_{s-1 \ \text{factors}}.
\end{equation}
If we act with the operator (\ref{S657}) on vector $v_{k\Lambda_0}$, we get 
\begin{align*} 
\left(1\otimes \cdots \otimes  1 \otimes  e_{\omega_{1}} \otimes 1 \otimes \cdots \otimes 1\right)(v_{k\Lambda_0})
=v_{\Lambda_{0}}\otimes \cdots \otimes v_{\Lambda_{0}}\otimes  v_{\Lambda_{1}}\otimes  \underbrace{v_{\Lambda_{0}} \otimes \cdots \otimes    v_{\Lambda_{0}}}_{s-1 \ \text{factors}}.
\end{align*}
Now it follows that we can write (\ref {glupaTanja}) as 
\begin{equation*}
\cdots \otimes x_{n^{(s)}_{r^{(2s)}_{2},2}\alpha_{2}}(z_{r_{2}^{(2s-1)},2})\cdots           x_{n^{(s)}_{r^{(2s)}_{2},2}\alpha_{2}}(z_{r^{(2s)}_{2},2})\cdots  x_{n^{(s)}_{1,2}\alpha_{2}}(z_{1,2})\end{equation*}
\begin{equation*}
x_{n^{(s)}_{r_{1}^{(k)},1}\alpha_{1}}(z_{r_{1}^{(k)},1})z_{r_{1}^{(k)},1}\cdots x_{\alpha_{1}}(z_{1,1})z_{1,1}v_{\Lambda_{0}}\cdots.\end{equation*}
By taking the corresponding coefficients, we have
\begin{equation*}
(1\otimes\cdots \otimes A_{\omega_{1}} \otimes 1 \otimes \cdots \otimes 1)\pi_{\mathfrak{R}}bv_{k\Lambda_0}=(1\otimes\cdots \otimes      e_{\omega_{1}} \otimes 1 \otimes \cdots \otimes 1)\pi_{\mathfrak{R}}b^{+}v_{k\Lambda_0}
\end{equation*}
where the monomial $b^{+}$:
\begin{align*}
b^{+}=b^+(\alpha_{2})b^{+}(\alpha_{1}),
\end{align*}
is such that 
\begin{align*}
b^+(\alpha_{2})&=b(\alpha_{2}),&\\
b^+(\alpha_{1})&=x_{n_{r_{1}^{(1)},1}\alpha_{1}}(m_{r_{1}^{(1)},1}+1)\cdots x_{s\alpha_{1}}(m_{1,1}+1)&\\
&=x_{n_{r_{1}^{(1)},1}\alpha_{1}}(m^{+}_{r_{1}^{(1)},1})\cdots x_{s\alpha_{1}}(m^{+}_{1,1}).&
\end{align*}

\subsection{Operator \texorpdfstring{$e_{\alpha_{1}}$}{ealpha1}}\label{S66}
For the simple root $\alpha_1 \in \Pi$, we define on the level $1$ standard module $L(\Lambda_0)$, the ``Weyl group translation'' operator  $e_{\alpha_1}$ by
\begin{equation*}
 e_{\alpha_1}=\exp\  x_{-\alpha_1}(1)\exp\  (- x_{\alpha_1}(-1))\exp\  x_{-\alpha_1}(1) \exp\ x_{\alpha_1}(0)\exp\   (-x_{-\alpha_1}(0))\exp\ x_{\alpha_1}(0),\end{equation*}
(cf. \cite{K}). On the standard $\widehat{\mathfrak{g}}$-module $L(\Lambda_0)$, we have
\begin{lem}
\begin{align}\label{S666}
x_{\alpha_1}(j)e_{\alpha_1}&=e_{\alpha_1}x_{\alpha_1}(j+2),\\
\label{S669}
x_{\alpha_2}(j)e_{\alpha_1}&=e_{\alpha_1}x_{\alpha_2}(j-1)
\end{align}
for every $j \in \mathbb{Z}$. 
\end{lem}
\begin{proof}
Expressions (\ref{S666}) and (\ref{S669}) follow from 
\begin{equation*}
(\exp \ a)b (\exp \ (-a))=( \exp \ (\text{ad} \ a))(b),
\end{equation*} 
for Lie algebra elements $a$ and $b$ such that $(\text{ad} \ a)^tb=0$ for some $t$ (cf. \cite{K}). 
\end{proof}
We will use the following expressions which follow from (\ref{S666}) and (\ref{S669}):
\begin{align}\label{mama1}
x_{\alpha_1}(z)e_{\alpha_1}&=z^2e_{\alpha_1}x_{\alpha_1}(z)\\
\label{mama2}
x_{\alpha_2}(z)e_{\alpha_1}&=z^{-1}e_{\alpha_1}x_{\alpha_2}(z).
\end{align}
It is easy to see that we have: 
\begin{lem}\label{mama}
\begin{equation}\label{S662}
e_{\alpha_1}v_{\Lambda_0}=-x_{\alpha_{1}}(-1)v_{\Lambda_{0}}
\end{equation}
\end{lem}
\begin{flushright}
$\square$
\end{flushright}
Set 
\begin{equation*}
1\otimes\cdots  \otimes 1\otimes \underbrace{e_{\alpha_{1}} \otimes e_{\alpha_{1}} \otimes \cdots \otimes e_{\alpha_{1}}}_{s \ \text{factors}}
\end{equation*}
where $s \leq k$. 
From Lemma \ref{mama}, it now follows
\begin{align*} 
&\left(1\otimes\cdots  \otimes 1\otimes e_{\alpha_{1}} \otimes e_{\alpha_{1}}\otimes  \cdots \otimes e_{\alpha_{1}}\right)v_{k\Lambda_0}&\\
=&(-1)^sv_{\Lambda_{0}}\otimes
 \cdots\otimes v_{\Lambda_{0}}\otimes \underbrace{x_{\alpha_{1}}(-1)v_{\Lambda_{0}}\otimes      x_{\alpha_{1}}(-1)v_{\Lambda_{0}}\otimes \cdots \otimes  x_{\alpha_{1}}(-1)v_{\Lambda_{0}}}_{s \ \text{factors}}&
\end{align*}
Let $b$ be a monomial   
\begin{align}\label{S6610}
b&=b(\alpha_{2})b(\alpha_{1})x_{s\alpha_{1}}(-s)&\\
\nonumber
&=x_{n_{r^{(1)}_{2},2}\alpha_{2}}(m_{r^{(1)}_{2},2})\cdots     x_{n_{1,2}\alpha_{2}}(m_{1,2})x_{n_{r^{(1)}_{1},1}\alpha_{1}}(m_{r^{(1)}_{1},1})\cdots  x_{n_{2,1}\alpha_{1}}(m_{2,1})x_{s\alpha_{1}}(-s),&
\end{align}
of dual-charge-type 
\begin{equation*}
\mathfrak{R}=\left(r^{(1)}_{2},\ldots, r^{(2k)}_{2};r^{(1)}_{1},\ldots, r_1^{(s)},0 \ldots, 0\right).
\end{equation*}
As in section \ref{ss:proj}, let $\pi_{\mathfrak{R}}$ be the projection of principal subspace $W_{L(\Lambda_{0})}\otimes \cdots\otimes W_{L(\Lambda_{0})}$ on the vector space  
\begin{equation*}
{W_{L(\Lambda_{0})}}_{(\mu^{(k)}_{2};0)}\otimes \cdots \otimes {W_{L(\Lambda_{0})}}_{(\mu^{(s)}_{2};r_{1}^{(s)})}\otimes \cdots\otimes {W_{L(\Lambda_{0})}}_{(\mu^{(1)}_{2};r_{1}^{(1)})}. \end{equation*}
The projection \begin{equation*}
\pi_{\mathfrak{R}}b\left(v_{\Lambda_{0}}\otimes \ldots\otimes v_{\Lambda_{0}}\right)\end{equation*}
of the monomial vector $b\left(v_{\Lambda_{0}}\otimes
 \ldots \otimes  v_{\Lambda_{0}}\right)$ is a coefficient of the generating function
\begin{equation*}\pi_{\mathfrak{R}}x_{n_{r_{2}^{(1)},2}\alpha_{2}}(z_{r_{2}^{(1)},2})\cdots
 x_{n_{1,2}\alpha_{2}}(z_{1,2})x_{n_{r_{1}^{(1)},1}\alpha_{1}}(z_{r_{1}^{(1)},1}) 
\cdots  x_{n_{2,1}\alpha_{1}}(z_{2,1})\end{equation*}
\begin{equation*}\left(v_{\Lambda_{0}}\otimes
 \cdots \otimes  v_{\Lambda_{0}}\otimes x_{\alpha_{1}}(-1)v_{\Lambda_{0}}\otimes \cdots \otimes  x_{\alpha_{1}}(-1)v_{\Lambda_{0}}\right)\end{equation*}
\begin{equation*}=Cx_{n^{(k)}_{r^{(2k-1)}_{2},2}\alpha_{2}}(z_{r_{2}^{(2k-1)},2}) \cdots 
 x_{n^{(k)}_{r^{(2k)}_{2},2}\alpha_{2}}(z_{r^{(2k)}_{2},2})\cdots
 x_{n^{(k)}_{1,2}\alpha_{2}}(z_{1,2})v_{\Lambda_{0}}\end{equation*}
\begin{equation*}\otimes \cdots \otimes \end{equation*}
\begin{equation*}\otimes x_{n^{(s)}_{r^{(2s-1)}_{2},2}\alpha_{2}}(z_{r_{2}^{(2s-1)},2}) \cdots 
 x_{n^{(s)}_{r^{(2s)}_{2},2}\alpha_{2}}(z_{r^{(2s)}_{2},2})\cdots  x_{n^{(s)}_{1,2}\alpha_{2}}(z_{1,2})\end{equation*}
\begin{equation*} x_{n^{(s)}_{r^{(s)}_{1},1}\alpha_{1}}(z_{r_{1}^{(s)},1})\cdots    x_{n^{(s)}_{2,1}\alpha_{1}}(z_{2,1})e_{\alpha_{1}}v_{\Lambda_{0}}\end{equation*}
\begin{equation*}\otimes \cdots \otimes \end{equation*}
\begin{equation*}\otimes x_{n_{r^{(1)}_{2},2}^{(1)}\alpha_{2}}(z_{r_{2}^{(1)},2})\cdots 
 x_{n_{r^{(2)}_{2},2}^{(1)}\alpha_{2}}(z_{r_{2}^{(2)},2})\cdots x_{n_{2,2}^{(1)}\alpha_{2}}(z_{2,2})
x_{n_{1,2}^{(1)}\alpha_{2}}(z_{1,2})\end{equation*}
\begin{equation*}x_{n_{r^{(1)}_{1},1}^{(1)}\alpha_{1}}(z_{r_{1}^{(1)},1})\cdots   x_{n{_{2,1}^{(1)}\alpha_{1}}}(z_{2,1})e_{\alpha_{1}}v_{\Lambda_{0}},\end{equation*}
(see (\ref{projekcija})).  We shift operator $1\otimes\ldots \otimes e_{\alpha_{1}} \otimes e_{\alpha_{1}} \otimes \ldots \otimes        e_{\alpha_{1}}$ all the way to the left using commutation relations (\ref{mama1}) and (\ref{mama2}) 
\begin{equation*}
(1 \otimes \cdots \otimes  e_{\alpha_{1}} \otimes e_{\alpha_{1}} \otimes \cdots \otimes  e_{\alpha_{1}})\pi_{\mathfrak{R}'} b'\left(v_{\Lambda_{0}}\otimes \ldots\otimes v_{\Lambda_{0}}\right), 
\end{equation*} 
where 
\begin{equation*}
\mathfrak{R}'=\left(r^{(1)}_{2},\ldots, r^{(2s)}_{2};r^{(1)}_{1}-1,\ldots, r_1^{(s)}-1\right)
\end{equation*}
and
\begin{align*}
b'&=b'(\alpha_{2})b'(\alpha_{1})&\\
&=x_{n_{r^{(1)}_{2},2}\alpha_{2}}(m_{r^{(1)}_{2},2}-n^{(1)}_{r_{2}^{(1)},2}-\cdots-n^{(s)}_{r_{2}^{(1)},2})&\\
& \ \ \ \ \ \ \ \ \ \ \ \ \ \ \ \ \ \ \ \ \ \ \ \ \ \  \cdots x_{n_{1,2}\alpha_{2}}(m_{1,2}-n^{(1)}_{1,2}-\cdots-n^{(s)}_{1,2})&\\
& \ \ \ \ \ \ \ \ \ \ \ \ \ \ \ \ \ \ \ \ \ \ \ \ \ \ \  x_{n_{r^{(1)}_{1}\alpha_{1}}}(m_{r^{(1)}_{1},1}+2n_{r^{(1)}_{1}})\cdots   x_{n_{2,1}\alpha_{1}}(m_{2,1}+2n_{1,2})&\\
&=x_{n_{r^{(1)}_{2},2}\alpha_{2}}(m'_{r^{(1)}_{2},2})\cdots 
 x_{n_{1,2}\alpha_{2}}(m'_{1,2})x_{n_{r^{(1)}_{1}\alpha_{1}}}(m'_{r^{(1)}_{1},1})\cdots x_{n_{2,1}\alpha_{1}}(m'_{2,1}).&
\end{align*}
In the proof of linear independence, we use the following proposition 
\begin{prop}\label{S66P}
Let $b$ (\ref{S6610}) be an element of the set $B_{W_{L(k\Lambda_{0})}}$. Then the monomial $b'$ is an element of the set $B_{W_{L(k\Lambda_{0})}}$. 
\end{prop}
\begin{proof}
The proposition follows by considering the possible situation for $n_{p,i}$ $2\leq p \leq r^{(1)}_{1}$ and $ 1\leq p \leq r^{(1)}_{2}$ from which it follows that $m_{p,i}$ comply the defining conditions of the set $B_{W_{L(k\Lambda_{0})}}$:
\begin{enumerate}
\item For $n_{p,1}=\bar{s}\leq s$, we have
\begin{align*} 
m'_{p,1}&=m_{p,1}+2\bar{s}\\
&\leq -\bar{s}- 2 (p-1)\bar{s} +2\bar{s}\\
&=-\bar{s}- 2 (p-2)\bar{s}
\end{align*}
and 
\begin{align*}
m'_{p+1,1}&=m_{p+1,1}+2\bar{s}\\
& \leq  -2\bar{s}+m_{p,1}+2\bar{s}\\
&= m'_{p,1} -2\bar{s} 
\ \ \ \ \ \ \ \text{for} \ \ \ n_{p+1,1}=n_{p,1}.
\end{align*}
\item For $n_{p,2}\geq 2s$, we have
\begin{align*} 
m'_{p,2}&=  m_{p,2}-2s\\
 &\leq  -n_{p,2} 
- \sum_{p>p'>0} 2 \ \text{min} \left\{n_{p,2},n_{p',2}\right\}+  
\sum^{r^{(1)}_{1}}_{q=1}\text{min} \left\{n_{p,2},2n_{q,1}\right\}-2s\\
 &=  -n_{p,2}
- \sum_{p>p'>0} 2 \ \text{min} \left\{n_{p,2},n_{p',2}\right\}+ 
 \sum^{r^{(1)}_{1}}_{q=2}\text{min} \left\{n_{p,2},2n_{q,1}\right\}
\end{align*}
and
\begin{align*}
m'_{p+1,2}&= m_{p+1,2} -2s\\
 & \leq m_{p,2} -2n_{p,2} -2s\\ 
&= m'_{p,2} -2n_{p,2}
 \ \ \ \ \ \text{for} \ \ \ n_{p-1,2}=n_{p,2}.
\end{align*}
\item For $n_{p,2}< 2s$, we have:
\begin{align*}
m'_{p,2}&= m_{p,2}-n_{p,2}\\
 &\leq  -2n_{p,2}  
 - \sum_{p>p'>0} 2 \ \text{min} \left\{n_{p,2},n_{p',2}\right\}+ 
 \sum^{r^{(1)}_{1}}_{q=1}\text{ min} \left\{n_{p,2},2n_{q,1}\right\}\\
 &=  -n_{p,2}
- \sum_{p>p'>0} 2 \ \text{min} \left\{n_{p,2},n_{p',2}\right\}+ 
\sum^{r^{(1)}_{1}}_{q=2}\text{ min} \left\{n_{p,2},2n_{q,1}\right\}
\end{align*}
and
\begin{align*}
m'_{p+1,2} &= m_{p+1,2} -n_{p,2} \\
&\leq  m_{p,2} -3n_{p,2}\\
&= m'_{p,2} -2n_{p,2}
 \ \ \ \ \ \text{for} \ \ \ n_{p+1,2}=n_{p,2}.
\end{align*}
\end{enumerate}
\end{proof}

\subsection{The proof of linear independence of the set \texorpdfstring{$\mathfrak{B}_{W_{L(k\Lambda_{0})}}$}{BWL(kLambda0)}}\label{S67}
In this section, we prove the following theorem: 
\begin{thm}\label{S67T1}
The set $\mathfrak{B}_{W_{L(k\Lambda_0)}}$ forms a basis for the principal subspace $W_{L\left(k\Lambda_{0}\right)}$ of $L\left(k\Lambda_{0}\right)$.
\end{thm}
\begin{proof} 
Since the set $\mathfrak{B}_{W_{L(k\Lambda_0)}}$ spans the principal subspace $W_{L\left(k\Lambda_{0}\right)}$, we have to prove the linear independence. We prove linear independence of the set $\mathfrak{B}_{W_{L(k\Lambda_0)}}$ by induction on charge-type of monomials $b \in \mathfrak{B}_{W_{L(k\Lambda_0)}}$. 
\par Assume that we have
\begin{equation}\label{S6752}
\sum_{a \in A}
c_{a}b_av_{k\Lambda_{0}}=0, 
\end{equation}
where $A$ is a finite non-empty set and 
\begin{align*}
b_a\in  B_{W_{L(k\Lambda_0)}}.
\end{align*}
Assume that all $b_a$ are the same color-type $(r_{2},r_{1})$. Let $b$ be the smallest monomial in the linear lexicographic ordering ``$<$''
\begin{align*} 
b&=b(\alpha_{2})b(\alpha_{1})\\
&=x_{n_{r_{2}^{(1)},2}\alpha_{2}}(m_{r_{2}^{(1)},2})\cdots
 x_{n_{1,2}\alpha_{2}}(m_{1,2})x_{n_{r_{1}^{(1)},1}\alpha_{1}}(m_{r_{1}^{(1)},1})\cdots x_{n_{1,1}\alpha_{1}}(-j),
\end{align*}
of charge-type 
\begin{equation}\label{S675444}
\left(n_{r_{2}^{(1)},2}, \ldots , 
 n_{1,2}; n_{r_{1}^{(1)},1}, \ldots ,n_{1,1}\right)
\end{equation}
such that $c_{a}\neq 0$ and such that, for every $b_{a}$ of charge-type (\ref{S675444}), 
we have
\begin{equation*}m_{1,1}\geq -j,\end{equation*}
where $m_{1,1}$ is energy of quasi-particle $x_{n_{1,1}\alpha_1}(m_{1,1})$ of monomial $b_{a}$.  Denote by 
\begin{equation*}
\mathfrak{R}=\left( r_{2}^{(1)},\ldots , r_{2}^{(2k)}; r_{1}^{(1)},\ldots , r_{1}^{(n_{1,1})}\right),
\end{equation*}
the dual-charge-type of $b$ and by $\mathfrak{C}$ the set of all monomial vectors $b_{\mathfrak{R}}v_{k\Lambda_0}$ from (\ref{S6752}) such that monomials $b_{\mathfrak{R}}$ have a charge-type as (\ref{S675444}). 
For every $1\leq t\leq k$ such that
\begin{equation*}
\mu^{(t)}_{2}=r^{(2t)}_{2}+ r^{(2t-1)}_{2}
\end{equation*}
let  $\pi_{\mathfrak{R}}$ be the projection of $\underbrace{W_{L({\Lambda}_{0})}\otimes \ldots\otimes W_{L({\Lambda}_{0})}}_{k \ \text{factors}}$ on the vector space 
\begin{equation*}
{W_{L(\Lambda_{0})}}_{(\mu^{(k)}_{2};0)}\otimes \ldots  \otimes {W_{L(\Lambda_{0})}}_{(\mu^{(n_{1,1}+1)}_{2}; 0)}\otimes {W_{L(\Lambda_{0})}}_{(\mu^{(n_{1,1})}_{2};r_{1}^{(n_{1,1})})}\otimes  \cdots {W_{L(\Lambda_{0})}}_{(\mu^{(1)}_{2};r_{1}^{(1)})}.
\end{equation*}
By Lemma \ref{S62L}, $\pi_{\mathfrak{R}}$ maps to zero all monomial vectors $b_av_{k\Lambda_0}$ such that $b_a$ has a larger charge-type in the linear lexicographic ordering ``$<$'' than (\ref{S675444}). So, in (\ref{S6752}) we have a projection of $b_{a}v_{k\Lambda_0}$, where $b_{a}$ are of charge-type (\ref{S675444})
\begin{equation}\label{S6733}
\sum_{a} c_a\pi_{\mathfrak{R}}b_{a}v_{k\Lambda_{0}}=0, 
\end{equation}
On (\ref{S6733}), we act with
\begin{equation*}
1\otimes\ldots \otimes A_{\omega_{1}} \otimes \underbrace{1 \otimes \ldots \otimes 1}_{n_{1,1}-1 \ \text{factors}}.
\end{equation*}
From \ref{ss:intert} and \ref{ss:curr} follows
\begin{align*}
1\otimes\ldots \otimes A_{\omega_{1}} \otimes 1 \otimes \ldots \otimes 1\left(\sum_{a}
c_{a}\pi_{\mathfrak{R}}b_{a}v_{k\Lambda_{0}}\right)\\
=(1\otimes\ldots \otimes e_{\omega_{1}} \otimes 1 \otimes \ldots \otimes 1)\left(\sum_{a}
c_{a}\pi_{\mathfrak{R}}b_{a}^{+}v_{k\Lambda_{0}}\right).\end{align*}
After leaving out the invertible operator $1\otimes\ldots \otimes
e_{\omega_{1}} \otimes \underbrace{1 \otimes \ldots \otimes 1}_{n_{1,1}-1 \ \text{factors}}$, we get
\begin{equation*} 
\sum_{a} c_{a}\pi_{\mathfrak{R}}b_{a}^{+}v_{k\Lambda_{0}}=0,
\end{equation*}
where
\begin{equation*}
b_{{a}}^{+}=b_{a}^{+}(\alpha_{2})b_{a}^{+}(\alpha_{1}) \in B_{W_{L(k\Lambda_0)}}
\end{equation*}
are the same charge-type as $b_{a}$ in (\ref{S6733}). 

We act with $1\otimes\ldots \otimes A_{\omega_{1}} \otimes \underbrace{1 \otimes \ldots \otimes 1}_{n_{1,1}-1 \ \text{factors}}$, until $j$ becomes $- n_{1,1}$. In that case, we get 
\begin{equation*}
\sum_{a}
c_{a}\pi_{\mathfrak{R}}b_{a}(\alpha_{2})b_{a}^{+}(\alpha_{1})x_{n_{1,1}\alpha_{1}}(-n_{1,1})v_{k\Lambda_{0}}=0
\end{equation*}
where $b_{a}^{+}(\alpha_{1})x_{n_{1,1}\alpha_{1}}(-n_{1,1})$ is of color $i=1$ and \begin{equation*}b_{\mathfrak{R}}(\alpha_{2})b_{\mathfrak{R}}^{+}(\alpha_{1})x_{n_{1,1}\alpha_{1}}(-n_{1,1})v_{k\Lambda_{0}} \in \mathfrak{C},\end{equation*}
that is 
\begin{equation*}b_{\mathfrak{R}}(\alpha_{2})b_{\mathfrak{R}}^{+}(\alpha_{1})x_{n_{1,1}\alpha_{1}}(-n_{1,1})v_{k\Lambda_{0}} \in \mathfrak{B}_{W_{L(k\Lambda_0)}}.\end{equation*} 
From the subsection \ref{S66} follows 
\begin{align*}
&\pi_{\mathfrak{R}}b(\alpha_{2})b^{+}(\alpha_{1})x_{n_{1,1}\alpha_{1}}(-n_{1,1})v_{k\Lambda_{0}}&\\
&=(1\otimes\cdots 1\otimes  e_{\alpha_{1}}\otimes e_{\alpha_{1}} \cdots \otimes   e_{\alpha_{1}})b'(\alpha_{2})b'(\alpha_{1})v_{k\Lambda_{0}},&
\end{align*} where $b'(\alpha_{2})b'(\alpha_{1})$ does not have a quasi-particle of charge $n_{1,1}$. 
 $b'(\alpha_{2})b'(\alpha_{1})$ is of dual-charge-type
\begin{equation*}
\mathfrak{R}^{-}=\left( r_{2}^{(1)}, \ldots ,r_{2}^{(2k)}; r_{1}^{(1)}-1, \ldots , r_{1}^{(n_{1,1})}-1\right),
 \end{equation*}
and charge-type
\begin{equation*}
\left(n_{r_{2}^{(1)},2}, \ldots , n_{1,2}; n_{r_{1}^{(1)},1}, \ldots 
,n_{2,1}\right).\end{equation*}
such that\begin{equation*}
\left(n_{r_{2}^{(1)},2}, \ldots , n_{1,2}; n_{r_{1}^{(1)},1}, \ldots ,n_{2,1}\right) < 
\left(n_{r_{2}^{(1)},2}, \ldots , n_{1,2}; n_{r_{1}^{(1)},1}, \ldots ,n_{2,1},n_{1,1}\right).
\end{equation*} 
From Proposition \ref{S66P}, it follows that we get elements from the set $\mathfrak{B}_{W_{L(k\Lambda_0)}}$.
\par We apply the described processes, until we get monomial ``colored'' only with color $i=2$. Thus by the consideration at the end of subsection \ref{ss:proj}, we have $c_{a}=0$ and the desired theorem follows.
\end{proof}

\subsection{The proof of linear independence of the set \texorpdfstring{$\mathfrak{B}_{W_{N(k\Lambda_{0})}}$}{BWN(kLambda0)}}
We consider the surjective map of $\mathcal{L}(\mathfrak{n}_{+})$-modules
\begin{align*}
\rho_0:U(\mathcal{L}(\mathfrak{n}_{+}))&\rightarrow  W_{N(k{\Lambda}_{0})},\\
b&\mapsto  bv_{N(k{\Lambda}_{0})}. 
\end{align*}
The restriction of $\rho_0$ to $U(\mathcal{L}(\mathfrak{n}_{+})_{<0})$   
\begin{align*}
U(\mathcal{L}(\mathfrak{n}_{+})_{<0})&\rightarrow  W_{N(k{\Lambda}_{0})}\\
b&\mapsto  bv_{N(k{\Lambda}_{0})}&
\end{align*}
is an isomorphism of $\mathcal{L}(\mathfrak{n}_{+})_{<0}$-modules. Let 
\begin{equation*}
J=U(\mathcal{L}(\mathfrak{n}_{+})){\mathcal{L}(\mathfrak{n}_{+})}_{\geq 0}
\end{equation*}
be the left ideal in $U(\mathcal{L}(\mathfrak{n}_{+}))$. Since the ideal $J$ lies in the kernel of $\rho_0$, we can factorize $\rho_0$ 
to a quotient map
\begin{equation*}
\rho:W\rightarrow W_{N(k{\Lambda}_{0})},
\end{equation*} 
where
\begin{equation*}
W=U(\mathcal{L}(\mathfrak{n}_{+}))/J.
\end{equation*}
\begin{ax}\label{S8N1}
$W$ does not depend on the central element $c$, since the root vectors $x, y$ $ \in \mathfrak{n}_{+}$ are orthogonal relative to invariant  simetric bilinear form $\left\langle \cdot,\cdot\right\rangle$ (cf. \cite{H}).
\end{ax}
From Poincar\'{e}-Birkhoff-Witt theorem, it follows 
\begin{equation}\label{84}
W \cong U(\mathcal{L}(\mathfrak{n}_{+})_{<0})\cong S(\mathcal{L}(\mathfrak{n}_{+})_{<0}). 
\end{equation}
Denote by $\pi$ the projection 
\begin{equation*}
\pi: U(\mathcal{L}(\mathfrak{n}_{+})) \rightarrow W,\end{equation*}
\begin{equation*}
\pi(b)=b+J.\end{equation*}
We have 
\begin{equation*}
 \rho(\pi(b))=bv_{N(k\Lambda_0)}.\end{equation*}
From the above consideration, we may conclude that the map $\rho$ is isomorphism of $\mathcal{L}(\mathfrak{n}_{+})$-modules.
 \par $W$ is $\mathbb{Z}$-graded vector space (see (\ref{84})). Let
\begin{equation*}
1 \in \mathbb{C} \subset W.
\end{equation*}
We have  
\begin{equation*}
W = \bigoplus_{n\geq 0}W_{(n)},
\end{equation*}
where  $W_{(n)}$ is spanned by the monomials
\begin{align*}
x_{1}(m_1)\cdots x_{t}(m_t)
\end{align*}
for $t\geq 0$, $x_t \in \mathfrak{n}_+$, with 
$m_1+ \cdots +m_t=-n$, 
$m_j\leq  -1$ 
and $x_{j} \in \mathfrak{n}_{+}$ for every $j=\left\{1, \ldots ,t\right\}$. 
\par One can obtain a basis of $W$ by using an ordered basis of $\mathcal{L}(\mathfrak{n}_{+})$. We obtain a basis of $W$ in terms of quasi-particles. By using this basis, we show that the set $\mathfrak{B}_{W_{N(k\Lambda_0)}}$ is the basis of the principal subspace $W_{N(k\Lambda_0)}$.
\par Define the set $B$:
\begin{equation*}
B= \bigcup_{\substack{0 \leq n_{r_{1}^{(1)},1}\leq \ldots \leq n_{1,1} \\ 0\leq n_{r_{2}^{(1)},2}\leq \ldots \leq n_{1,2}}}\left(\text{or, equivalently,} \ \ \bigcup_{\substack{r_{1}^{(1)}\geq r_{1}^{(2)}\geq \ldots \geq 0\\ r_{2}^{(1)}\geq r_{2}^{(2)}\geq \ldots \geq 0}}\right)\end{equation*}
\begin{equation*}
\left\{b=b(\alpha_{2})b(\alpha_{1})=x_{n_{r_{2}^{(1)},2}\alpha_{2}}(m_{r_{2}^{(1)},2})\cdots x_{n_{1,1}\alpha_{1}}(m_{1,1})\right.:\end{equation*}
\begin{align}\nonumber
\left|
\begin{array}{l}
m_{p,1}\leq  -n_{p,1} - \sum_{p>p'>0} 2 \ \mathrm{min}\{n_{p,1}, 
n_{p',1}\}, \  1\leq  p\leq  r_{1}^{(1)};\\
\nonumber 
m_{p+1,1}\leq  m_{p,1}-2n_{p,1} \  \mathrm{if} \ n_{p,1}=n_{p+1,1}, \  1\leq  p\leq  r_{1}^{(1)}-1;\\
\nonumber 
m_{p,2}\leq  -n_{p,2}+ \sum_{q=1}^{r_{1}^{(1)}}\mathrm{min}\left\{2n_{q,1},n_{p,2}\right\}- \sum_{p>p'>0} 
2\mathrm{min}\{n_{p,2}, n_{p',2}\}, \ 1\leq  p \leq  r_{2}^{(1)};\\
m_{p+1,2} \leq   m_{p,2}-2n_{p,2} \  \mathrm{if} \ n_{p+1,2}=n_{p,2}, \ 1\leq  p \leq  r_{2}^{(1)}-1
\end{array}
\right\}.\end{align}
Obviously $\rho$ maps the set $B$ to a set  
\begin{equation*}
\mathfrak{B}_{W_{N(k{\Lambda}_{0})}}=\left\{bv_{N(k{\Lambda}_{0})}\left| b \in B\right.\right\}.
\end{equation*}
In the same way as the Proposition \ref{prop:S21}, we can prove
\begin{lem}\label{L3}
The set $B$ spans $W$.
\end{lem}
\begin{flushright}
$\square$
\end{flushright}
\par Fix $k' \in \mathbb{N}$ such that $k' \neq k$. Let $\rho_{k'}$ be a map $$\rho_{k'}: W\rightarrow W_{N(k'{\Lambda}_{0})},$$ 
defined by
$$ \rho_{k'}(\pi(b))=bv_{N(k\Lambda_0)},$$
where $b \in U(\mathcal{L}(\mathfrak{n}_{+}))$ and $\pi: U(\mathcal{L}(\mathfrak{n}_{+})) \rightarrow W$ projection on $W$. Since  $\rho_{k'}$ is isomorphism of $\mathcal{L}(\mathfrak{n}_{+})$-modules, we have: 
\begin{cor} The sets $\mathfrak{B}_{W_{N(k{\Lambda}_{0})}}$ and $\mathfrak{B}_{W_{N(k'{\Lambda}_{0})}}$ are equivalent in the sense 
$$\rho_{k'}\left(\rho^{-1}\left(\mathfrak{B}_{W_{N(k{\Lambda}_{0})}}\right)\right)=\mathfrak{B}_{W_{N(k'{\Lambda}_{0})}}.$$ 
\end{cor}
\begin{flushright}
$\square$
\end{flushright}
\bigskip
\bigskip
\bigskip
 Now we can prove the Theorem \ref{thm:1}.

\begin{flushleft}
\textbf{\emph{Proof of the Theorem \ref{thm:1}:}}
\end{flushleft}
Since the set $\mathfrak{B}_{W_{N(k{\Lambda}_{0})}}$ spans the principal subspace $W_{N\left(k{\Lambda}_{0}\right)}$, all we have to prove is the linear independence of the set $\mathfrak{B}_{W_{N(k{\Lambda}_{0})}}$. Assume that 
\begin{equation}\label{S8T1}
\sum_{a \in A}
c_{a}b_av_{N(k\Lambda_{0})}=0, 
\end{equation}
where $A$ is a finite non-empty set. We assume that all monomials $b_a$ are the same color-type $(r_{2},r_{1})$. Let $b'$ be the largest monomial in (\ref{S8T1}) in the linear lexicographic ordering ``$<$'' (cf. (\ref{eq:S23})), such that $c'_{a} \neq 0$, of charge-type
 \begin{equation*}
 \left(n'_{r_{2}^{(1)},2}, \ldots, n'_{1,2};n'_{r_{1}^{(1)},1}, \ldots, n'_{1,1}\right)
 \end{equation*}
and dual-charge-type
\begin{equation*}
\left(r_{2}^{(1)},r_{2}^{(2)},\ldots , r_{2}^{(n'_{1,2})} ; r_{1}^{(1)}, r_{1}^{(2)}, \ldots , r_{1}^{(n'_{1,1})}  \right).
\end{equation*}
Since the map $\rho$ is isomorphism of $\mathcal{L}(\mathfrak{n}_{+})$-modules, we have
\begin{equation*}
\sum_{a \in A}
c_{a}b_a=0, 
\end{equation*}
where
\begin{align}
b_a\in  B.
\end{align}
Let $k'$ be a positive integer such that $$n'_{1,1}\leq k'$$ and $$\frac{1}{2}n'_{1,2}\leq k'.$$ 
Denote by  $W_{N\left(k'{\Lambda}_{0}\right)}=U(\mathcal{L}(\mathfrak{n}_{+}))v_{N(k'{\Lambda}_{0})}$ the principal subspace of $N\left(k'{\Lambda}_{0}\right)$. 
The map $\rho_{k'}$ 
$$\rho_{k'}: W\rightarrow W_{N(k'{\Lambda}_{0})}$$ 
maps the set $B$ to the set that spans $W_{N(k'{\Lambda}_{0})}$, so we have
\begin{equation*} 
\sum_{a \in A}
c_{a}b_av_{N(k'\Lambda_{0})}=0. 
\end{equation*}
With the surjective map 
$f_{k'\Lambda_0}:W_{N(k'{\Lambda}_{0})}\rightarrow W_{L(k'{\Lambda}_{0})}$, 
we get
\begin{equation*} 
\sum_{a \in A}
c_{a}b_av_{k'\Lambda_{0}}=0. 
\end{equation*}
From the Theorem \ref{S67T1}, we see that all the coefficients $c_a$ in (\ref{S8T1}) are equal to zero and our proof of linear independence is complete. 
\begin{flushright}
$\square$
\end{flushright}

\section{Character formulas}
The space $N(k{\Lambda}_{0})$ has certain gradings. The action of the operator $-d$ on $N(k{\Lambda}_{0})$ gives a grading by weight (see (\ref{eq:S11}) and (\ref{eq:S13})). The action of Cartan subalgebra $\mathfrak{h}$ on $N(k{\Lambda}_{0})$ provides the gradation on $N(k{\Lambda}_{0})$, which we call color-type gradation (see (\ref{eq:S22})).
\par We restrict these $\mathbb{Z}$-gradings to the principal subspaces $W_{L(k{\Lambda}_{0})}$ and $W_{N(k{\Lambda}_{0})}$.
Let $p_i$, $n_{p_i,i},\ldots , n_{1,i}$, (where $1\leq i \leq 2$), be non-negative integers.
For integers $m_{p_i,i}, \ldots, m_{1,i}$, the elements:
\begin{equation*} 
x_{n_{p_2,2}\alpha_{2}}(m_{p_2,2})\cdots x_{n_{1,2}\alpha_{2}}(m_{1,2})x_{n_{p_1,1}\alpha_{1}}(m_{p_1,1})\cdots
 x_{n_{1,1}\alpha_{1}}(m_{1,1})v_{k\Lambda_0}\end{equation*}
and
\begin{equation*}
x_{n_{p_2,2}\alpha_{2}}(m_{p_2,2})\cdots x_{n_{1,2}\alpha_{2}}(m_{1,2})x_{n_{p_1,1}\alpha_{1}}(m_{p_1,1})\cdots        x_{n_{1,1}\alpha_{1}}(m_{1,1})v_{N(k\Lambda_0)}
\end{equation*}
have weight
\begin{equation*} 
-m_{p_2,2}- \ldots - m_{1,2}-m_{p_1,1}- \ldots -m_{1,1}=-m.
\end{equation*}
Their color-type is
\begin{equation*}
(r_2, r_1),
\end{equation*} where
\begin{equation*} 
r_i=n_{p_{i},i}+ \ldots + n_{1,i}.\end{equation*}
\par Such vectors span a weight subspace which we denote as ${W_{L(k\Lambda_{0})}}_{(m,r_1,r_2)}$ and\\  ${W_{N(k\Lambda_{0})}}_{(m,r_1,r_2)}$.

\subsection{Characters of the principal subspace \texorpdfstring{$W_{L(k\Lambda_{0})}$}{WL(kLambda0)}}
Here we consider the graded dimensions of the principal subspace $W_{L(k\Lambda_{0})}$. Denote by $\text{ch} \ W_{L(k\Lambda_{0})}$ the     generating function of dimensions of homogeneous subspaces of $W_{L(k\Lambda_{0})}$:
 \begin{align}\label{S75}
 \text{ch} \ W_{L(k\Lambda_{0})}=\sum_{m,r_1,r_2\geq 0} 
\text{dim} \ {W_{L(k\Lambda_{0})}}_{(m,r_1,r_2)}q^{m}y^{r_1}_{1}y^{r_2}_{2},
\end{align}
where $q,\ y_1$ and $y_2$ are formal variables (cf. \cite{FLM}). This generating function is called the character of the principal         subspace $W_{L(k\Lambda_{0})}$.
\par We write down $\text{ch} \ W_{L(k\Lambda_{0})}$ in terms of the dual-charge-type para\-meters $r_i^{(s)}$. Therefore, we write         conditions in the definition of the set $\mathfrak{B}_{W_{L(k\Lambda_0)}}$ (\ref{SkupL}) in terms of $r_i^{(s)}$. More precisely, we use   the following expressions (\ref{S710}), (\ref{S711}), and (\ref{S712}) to determine the character of $W_{L(k\Lambda_{0})}$. It is easy to prove these expressions  by using induction on the level $k \in \mathbb{N}$ of the standard module $L(k\Lambda_0)$.
\begin{lem}\label{S7L1}
For the given color-type $(r_{2},r_{1})$, charge-type      
\begin{equation*}\left(n_{r_{2}^{(1)},2}, \ldots, n_{1,2};\right. \left. n_{r_{1}^{(1)},1}, \ldots , n_{1,1}\right)\end{equation*} and dual-charge-type   
\begin{equation*}
\left(r_{2}^{(1)},r_{2}^{(2)},\ldots ,r_{2}^{(2k)}; r_{1}^{(1)}, r_{1}^{(2)}, \ldots ,r_{1}^{(k)}\right), \end{equation*}
we have:
\begin{align}\label{S710}
\sum_{p=1}^{r^{(1)}_2}\sum_{q=1}^{r^{(1)}_1}\mathrm{min}\{n_{p,2},2n_{q,1}\}&=\sum_{s=1}^{k}r^{(s)}_1(r_2^{(2s-1)}+r_2^{(2s)}),&\\
\label{S711}\sum_{p=1}^{r_{1}^{(1)}} (\sum_{p>p'>0}2\mathrm{min} \{ n_{p,1},
n_{p',1}\}+n_{p,1})&= \sum_{s=1}^{k}r^{(s)^{2}}_{1},\\
\label{S712}\sum_{p=1}^{r_{2}^{(1)}} (\sum_{p>p'>0}2\mathrm{min} \{ n_{p,2},
n_{p',2}\}+n_{p,2})&= \sum_{s=1}^{2k}r^{(s)^{2}}_{2}.&
\end{align}
\end{lem}
\begin{flushright}
$\square$
\end{flushright}
\par For $r>0$ set 
\begin{equation*}\frac{1}{(q)_r}=\frac{1}{(1-q)(1-q^2)\cdots (1-q^r)}.\end{equation*}
For this formal power series, we have a combinatorial interpretation:
\begin{align}\label{S7K}
\frac{1}{(q)_r}=\sum_{j\geq 0}p_r(j)q^j,
\end{align}
where $p_r(j)$ is the number of partition of $j$ with most $r$ parts 
(cf. \cite{A}).
\par Now, from the definition of the set $\mathfrak{B}_{W_{L(k\Lambda_0)}}$ and (\ref{S710}), (\ref{S711}), (\ref{S712}), (\ref{S7K}) follows the character formula:
\begin{thm}
\begin{align}\nonumber 
&\mathrm{ch} \  W_{L(k\Lambda_{0})}&\\
\nonumber
= \sum_{\substack{r^{(1)}_{1}\geq \ldots \geq r^{(k)}_{1}\geq 0}}
&\frac{q^{r^{(1)^{2}}_{1}+\cdots +r^{(k)^{2}}_{1}}}{(q)_{r^{(1)}_{1}-r^{(2)}_{1}}\cdots (q)_{r^{(k)}_{1}}}y^{r_1}_{1}&\\
\nonumber
\sum_{r^{(1)}_{2}\geq \ldots \geq r^{(2k)}_{2}\geq  0}&\frac{q^{r^{(1)^{2}}_{2}+\cdots +r^{(2k)^{2}}_{2}-r_{1}^{(1)}(r_{2}^{(1)}+r_{2}^{(2)})
-\cdots -r_{1}^{(k)}(r_{2}^{(2k-1)}+r_{2}^{(2k)})}}{(q)_{r^{(1)}_{2}-r^{(2)}_{2}}\cdots (q)_{r^{(2k)}_{2}}}
y^{r_2}_{2}.&\\
\nonumber
&& \ \ \ \ \ \ \ \ \ \ \ \ \ \ \ \ \ \ \ \ \ \ \ \ \ \ \ \ \square
\end{align}
\end{thm}

\subsection{Characters of the principal subspace \texorpdfstring{$W_{N(k\Lambda_{0})}$}{WN(kLambda0)}}
In the same way as in the previous subsection, we define the character of the principal subspace $W_{N(k\Lambda_{0})}$ as 
 \begin{align*}  
 \mathrm{ch} \ W_{N(k\Lambda_{0})}=\sum_{m,r_1,r_2\geq 0} 
\mathrm{dim} \ {W_{N(k\Lambda_{0})}}_{(m,r_1,r_2)}q^{m}y^{r_1}_{1}y^{r_2}_{2},
\end{align*}
where $q,\ y_1$ and $y_2$ are formal variables. 
\par From the definition of set $\mathfrak{B}_{W_{N(k\Lambda_0)}}$ and (\ref{S710}), (\ref{S711}), (\ref{S712}), (\ref{S7K}) follows the character formula of the principal subspace $W_{N(k\Lambda_0)}$:
\begin{thm}
\begin{align}\label{KN2}
&\mathrm{ch} \  W_{N(k\Lambda_{0})}&\\
\nonumber
=\sum_{\substack{r^{(1)}_{1}\geq \ldots \geq r^{(u)}_{1}\geq 0\\ u\geq  0}}&\frac{q^{r^{(1)^{2}}_{1}+\cdots +r^{(u)^{2}}_{1}}}{(q)_{r^{(1)}_{1}-r^{(2)}_{1}}\cdots (q)_{r^{(u)}_{1}}}y^{r_1}_{1}&\\
\nonumber
\sum_{\substack{r^{(1)}_{2}\geq \ldots \geq r^{(v)}_{2}\geq 0\\ v \geq 0}}&\frac{q^{r^{(1)^{2}}_{2}+\cdots   +r^{(v)^{2}}_{2}-r_{1}^{(1)}(r_{2}^{(1)}+r_{2}^{(2)})
-\cdots -r_{1}^{(u)}(r_{2}^{(2u-1)}
 +r_{2}^{(2u)})-\cdots -r_{1}^{(v)}(r_{2}^{(2v-1)}
 +r_{2}^{(2v)})}}{(q)_{r^{(1)}_{2}-r^{(2)}_{2}}\cdots (q)_{r^{(v)}_{2}}}y^{r_2}_{2}.&
\end{align}
\end{thm}
\begin{flushright}
$\square$
\end{flushright}
We can determine the character of principal subspace $W_{N(k{\Lambda}_{0})}$ using the Po\-incar\'{e}-Birk\-hoff-Witt theorem.
\par We set $\left\{x_{\alpha_{1}},x_{\alpha_{2}},x_{\alpha_{1}+\alpha_{2}},x_{\alpha_1+2\alpha_2}\right\}$ as a basis of the Lie subalgebra $\mathfrak{n}_{+}$ and order these basis elements as follows: \begin{equation*}x_{\alpha_{2}}<x_{\alpha_{1}}<x_{\alpha_{1}+\alpha_{2}}<x_{\alpha_1+2\alpha_2}.\end{equation*}
Set $\left\{x_{\alpha}(m): \alpha \in R_+, m <0\right\}$, so that this set is a basis of the Lie algebra $\mathcal{L}(\mathfrak{n}_+)_{<0}$. We choose the following total order on this set: 
\begin{equation*}x(m)\leq y(m')\Leftrightarrow  x< y \quad \text{or} \quad x=y \quad \text{and} \quad m< m'.
\end{equation*}
 By Poincar\'{e}-Birkhoff-Witt theorem, we obtain the base of the universal
enveloping algebra $U(\mathcal{L}(\mathfrak{n}_+)_{<0})$:
\begin{align}\label{102}
&x_{\alpha_2}(m^1_1)\cdots  x_{\alpha_2}(m^{s_1}_1)x_{\alpha_1}(m^1_2)\cdots  x_{\alpha_1}(m^{s_2}_2)&\\
\nonumber
&\ \ \ \ \ \ \ \ \ \ x_{\alpha_1+\alpha_2}(m^1_3)\cdots x_{\alpha_1+\alpha_2}(m^{s_3}_3)x_{\alpha_1+2\alpha_2}(m^1_4)\cdots   x_{\alpha_1+2\alpha_2}(m^{s_4}_4),& 
\end{align} with $m_i^1\leq \cdots \leq m_i^{s_i}$, $s_i \in \mathbb{N}$ for $i=1,2,3,4$. \\ 
It follows that the subspace $U(\mathcal{L}(\mathfrak{n}_+)_{<0})_{(m,r_1,r_2)}$ has basis (\ref{102}), where  
\begin{equation*}(m,r_1,r_2)=(\sum_{i=1}^4\sum_{j=1}^{s_i}m^j_i, s_2+s_3+s_4, s_1+s_3+2s_4).\end{equation*}
The  bijection map 
\begin{align}\nonumber
U(\mathcal{L}(\mathfrak{n}_+)_{<0})&\rightarrow  W_{N(k{\Lambda}_{0})}\\
\nonumber
b&\mapsto  bv_{N(k{\Lambda}_{0})}
\end{align}
maps weighted subspace $U(\mathcal{L}(\mathfrak{n}_+)_{<0})_{(m,r_1,r_2)}$ on ${W_{N(k{\Lambda}_{0})}}_{(m,r_1,r_2)}$. Thus, we also have
\begin{align}\label{KN3}
\textrm{ch} \ W_{N(k\Lambda_{0})}
=\prod_{m > 0} \frac{1}{(1-q^my_1)}\frac{1}{(1-q^my_2)}\frac{1}{(1-q^my_1y_2)}\frac{1}{(1-q^my_1y_2^2)}.
\end{align}
Now from (\ref{KN2}) and (\ref{KN3}) follows a new identity of Rogers-Ramanujan's type:
\begin{thm}
\begin{align}\nonumber
\prod_{m > 0}& \frac{1}{(1-q^my_1)}\frac{1}{(1-q^my_2)}\frac{1}{(1-q^my_1y_2)}\frac{1}{(1-q^my_1y_2^2)}&\\
\nonumber
= \sum_{\substack{r^{(1)}_{1}\geq \ldots \geq r^{(u)}_{1}\geq 0\\ u \geq 
0}}&\frac{q^{r^{(1)^{2}}_{1}+\cdots +r^{(u)^{2}}_{1}}}{(q)_{r^{(1)}_{1}-r^{(2)}_{1}}\cdots (q)_{r^{(u)}_{1}}}y^{r_1}_{1}&\\
\nonumber
\sum_{\substack{r^{(1)}_{2}\geq \ldots \geq r^{(v)}_{2}\geq 0\\v \geq 
0}}&\frac{q^{r^{(1)^{2}}_{2}+\cdots +r^{(v)^{2}}_{2}-r_{1}^{(1)}(r_{2}^{(1)}+r_{2}^{(2)})
-\cdots -r_{1}^{(u)}(r_{2}^{(2u-1)}
+r_{2}^{(2u)})-\cdots -r_{1}^{(v)}(r_{2}^{(2v-1)}
+r_{2}^{(2v)})}}{(q)_{r^{(1)}_{2}-r^{(2)}_{2}}\cdots (q)_{r^{(v)}_{2}}}y^{r_2}_{2}&
\end{align}
\end{thm}
\begin{flushright}
$\square$
\end{flushright}


\begin{thebibliography}{9}
\bibitem[A]{A} G. E. Andrews, \textit{The theory of partitions}, Encyclopedia
of Mathematics and Its Applications, Vol. 2, Addison-Wesley, 1976. 
\bibitem[AKS]{AKS} E. Ardonne, R. Kedem and M. Stone, \textit{Fermionic characters of arbitrary highest-weight integrable
$\widehat{sl}_{r+1}$-modules}, Comm. Math. Phys. \textbf{264} (2006), 427-464; \url{ arXiv:math.RT/0504364}.
\bibitem[Cal1]{Cal1} C. Calinescu, \textit{Intertwining vertex operators and certain representations of $\widehat{sl(n)}$},
Commun. Contemp. Math. \textbf{10} (2008), 47-79; \url{arXiv:math:QA/0611534}.
\bibitem[Cal2]{Cal2} C. Calinescu, \textit{Principal subspaces of higher-level standard $\widehat{sl(3)}$-modules}, J. Pure
Appl. Algebra \textbf{210} (2007), 559-575; \url{arXiv:math.QA/0611540}.
\bibitem[CalLM1]{CalLM1} C. Calinescu, J. Lepowsky and A. Milas, \textit{Vertex-algebraic structure of the principal subspaces of
certain $A^{(1)}_{1}$-modules, I: level one case}, Int. J. Math. \textbf{19} (2008), 71-92; \url{arXiv.math.QA/0710.1527}.
\bibitem[CalLM2]{CalLM2} C. Calinescu, J. Lepowsky and A. Milas, \textit{Vertex-algebraic structure of the principal subspaces of
certain $A^{(1)}_{1}$ -modules, II: higher-level case}, J. Pure Appl. Algebra, \textbf{212} (2008), 1928-1950; \url{arXiv.math.QA/0710.1527}.
\bibitem[CalLM3]{CalLM3} C. Calinescu, J. Lepowsky and A. Milas, \textit{Vertex-algebraic structure of the principal subspaces of
level one modules for the untwisted affine Lie algebras of types A,D,E}, Journal of Algebra \textbf{323} (2010), 167-192; \url{ 	arXiv:math.QA/0908.4054}. 
\bibitem[CLM1]{CLM1} S. Capparelli, J. Lepowsky and A. Milas,\textit{ The Ro\-gers-Ra\-ma\-nu\-jan recur\-si\-on and
intertwining o\-perators}, Comm. Con\-tem\-po\-ra\-ry Math. \textbf{5} (2003), 947-966; \url{arXiv:math.QA/0211265}.
\bibitem[CLM2]{CLM2} S. Capparelli, J. Lepowsky and A. Milas, \textit{The Rogers-Selberg recursions, the Gordon-Andrews
identities and intertwining operators}, Ramanujan J. \textbf{12} (2006), 379-397; \url{arXiv:math.QA/0310080}.
\bibitem[DKKMM]{DKKMM} S. Dasmahapatra, R. Kedem, T. R. Klassen, B. McCoy  and E. Melzer:
\textit{Quasi-particles, conformal field theory and q-series.} Int. J. Mod. Phys.
\textbf{B7}, 3617 (1993); \url{arXiv:hep-th/9303013}.
\bibitem[DL]{DL} C. Dong and J. Lepowsky, \textit{Generalized vertex algebras and relative vertex operators}, Progress in
Mathematics  \textbf{112}, Birkh$\ddot{\text{a}}$user, Boston, 1993.
\bibitem[DLM]{DLM} C. Dong, H. Li and G. Mason, \textit{Simple currents and extensions of vertex operator
algebras}, Comm. Math. Physics \textbf{180} (1996), 671-707.
\bibitem[F]{F} E. Feigin, \textit{The PBW filtration}, Represent. Theory \textbf{13} (2009), 165-181; \url{arXiv:math.QA/0702797}.
\bibitem[FS]{FS} A. V. Stoyanovsky and B. L. Feigin, \textit{Functional models of the representations of
current algebras, and semi-infinite Schubert cells}, (Russian) Funktsional. Anal. i
Prilozhen. \textbf{28} (1994), no. 1, 68-90, 96; translation in Funct. Anal. Appl. \textbf{28} (1994),
no. 1, 55-72; preprint B. L. Feigin and A. V. Stoyanovsky, \textit{Quasi-particles models for the
representations of Lie algebras and geometry of flag manifold}; \url{arXiv:hep-th/9308079}.
\bibitem[FHL]{FHL} I. B. Frenkel, Y.-Z. Huang and J. Lepowsky,
\textit{On Axiomatic Approaches to Vertex Operator Algebras and Modules}, Memoirs of the Amer. Math. Soc. {\bf 104}, 1993.
\bibitem[FLM]{FLM} B. Frenkel, J. Lepowsky and A. Meurman,
{\em Vertex Operator Algebras and the Monster},
Pure and Appl. Math., \textbf{134}, Academic Press, Boston, 1988.
\bibitem[G]{G} G. Georgiev, \textit{Combinatorial constructions of modules for infinite-dimensional Lie algebras,
I. Principal subspace}, J. Pure Appl. Algebra \textbf{112} (1996), 247-286; \url{arXiv:hep-th/9412054}.
\bibitem[H]{H} J. Humphreys, \textit{Introduction to Lie Algebras and Their Representations}, Graduated
Texts in Mathematics, Springer-Verlag, New York, 1972. 
\bibitem[JP]{JP} M. Jerkovi\'{c} and M. Primc, \textit{Quasi-particle fermionic formulas for (k, 3)-admissible configurations}, to appear in Central European Journal of Mathematics; \url{arXiv:math.QA/1107.3900}.
\bibitem[K]{K} V. G. Kac, \textit{Infinite Dimensional Lie Algebras}, 3rd ed., Cambridge University Press, Cambridge, 
 1990.
\bibitem[LL]{LL} J. Lepowsky and H.-S. Li, \textit{Introduction to Vertex Operator Algebras and
Their Representations}, Progress in Math. Vol. \textbf{227}, Birkh$\ddot{\text{a}}$user,
Boston, 2003.
\bibitem[LP]{LP} J. Lepowsky and M. Primc, \textit{Standard modules for type one affine Lie algebras}, Lecture Notes in
Math. \textbf{1052} (1984), 194-251.
\bibitem[Li1]{Li1} H.-S. Li, \textit{Local systems of vertex operators, vertex superalgebras and modules},
J. Pure Appl. Alg. \textbf{109} (1996), 143-195; \url{arXiv:math.QA/9504022}.
\bibitem[Li2]{Li2} H.-S. Li, \textit{Certain extensions of vertex operator algebras of affine type}, Commun.
Math. Phys. \textbf{217} (2001), 653-696. 
\bibitem[MP]{MP} A. Meurman and M. Primc, \textit{Annihilating fields of standard mo\-dules
of $\widetilde{sl}(2, \mathbb{C})$ and combinatorial identities}, Memoirs Amer. Math. Soc. \textbf{137}, (1999);  \url{arXiv:math.QA/9806105}.
\bibitem[P]{P} M. Primc, \textit{Combinatorial basis of modules for affine Lie algebra $B_2^{(1)}$}, to appear in
Central European Journal of Mathematics; \url{arXiv:math.QA/ 1002.3535}.
\end{thebibliography}
\end{document}